\documentclass[11pt,reqno,a4paper]{article} 

\usepackage{anysize}
\marginsize{3.0cm}{3.0cm}{2.2cm}{2.2cm}

\usepackage[english]{babel}
\usepackage[centertags]{amsmath}
\usepackage{amsfonts,amssymb,amsthm}
\usepackage{hyperref}  					 %%%% buono per PDF	
\usepackage{mathrsfs}
\usepackage[sans]{dsfont}
\let\OLDthebibliography\thebibliography
\renewcommand\thebibliography[1]{
  \OLDthebibliography{#1}
  \setlength{\parskip}{0pt}
  \setlength{\itemsep}{0pt plus 0.3ex} }

\numberwithin{equation}{section}

\theoremstyle{plain}
\newtheorem{theorem}{Theorem}[section]

\newtheorem{lemma}[theorem]{Lemma}

\theoremstyle{definition}
\newtheorem{remarkbase}[theorem]{Remark}
\newenvironment{remark}{\pushQED{\qed} \remarkbase}{\popQED\endremarkbase}

\newcommand{\N}{{\mathbb N}}
\newcommand{\R}{{\mathbb R}}
\newcommand{\C}{{\mathbb C}}
\newcommand{\Z}{\mathbb Z}

\newcommand{\T}{{\mathbb T}}

\newcommand{\mA}{\mathcal{A}}
\newcommand{\mB}{\mathcal{B}}

\newcommand{\mF}{\mathcal{F}}

\newcommand{\mL}{\mathcal{L}}

\newcommand{\mS}{\mathcal{S}}

\newcommand{\mR}{\mathcal{R}}
\newcommand{\mM}{\mathcal{M}}
\newcommand{\mN}{\mathcal{N}}

\newcommand{\mT}{\mathcal{T}}

\renewcommand{\a}{\alpha}
\renewcommand{\b}{\beta}
\newcommand{\g}{\gamma}
\renewcommand{\d}{\delta}
\newcommand{\D}{\Delta}

\newcommand{\ph}{\varphi}
\newcommand{\lm}{\lambda}
\newcommand{\Lm}{\Lambda}

\newcommand{\om}{\omega}
\newcommand{\p}{\pi}
\newcommand{\s}{\sigma}
\renewcommand{\t}{\tau}
\renewcommand{\th}{\vartheta}

\newcommand{\gr}{\nabla}

\newcommand{\la}{\langle}
\newcommand{\ra}{\rangle}
\newcommand{\inv}{^{-1}}
\newcommand{\pa}{\partial}

\def\ba{\begin{aligned}}
\def\ea{\end{aligned}}
\def\beginm{\begin{multline}}
\def\endm{\end{multline}}

\def\xC{\C} %{\mathbf{C}}

\def\xZ{\Z} %{\mathbf{Z}}

\newcommand{\ttf}{\mN} % {\mathtt{f}}
\newcommand{\ttL}{\mathcal{L}}   %{\mathtt{L}}

\title{Exact controllability for quasi-linear perturbations of KdV}

\author{\small{Pietro Baldi, Giuseppe Floridia, Emanuele Haus}}
\date{} 
\begin{document}
\maketitle

\begin{small}
\textbf{Abstract.}
We prove that the KdV equation on the circle 
remains exactly controllable in arbitrary time with localized control, 
for sufficiently small data,
also in presence of quasi-linear perturbations, 
namely nonlinearities containing up to three space derivatives,
having a Hamiltonian structure at the highest orders.
We use a procedure of reduction to constant coefficients up to order zero (adapting \cite{BBM-Airy}),
classical Ingham inequality and HUM method 
to prove the controllability of the linearized operator. 
Then we prove and apply a modified version of the Nash-Moser implicit function theorems 
by H\"ormander \cite{Geodesy, Olli}. 
\emph{MSC2010:} 35Q53, 35Q93.
\end{small}

\bigskip

\emph{Contents.}
\ref{sec:intro} Introduction ---
\ref{sec:regu} Reduction of the linearized operator to constant coefficients ---
\ref{sec:obs} Observability ---
\ref{section:con} Controllability ---
\ref{sec:proof} Proofs ---
\ref{sec:WP} Appendix A. Well-posedness of linear operators ---
\ref{sec:NM} Appendix B. Nash-Moser theorem --- 
\ref{sec:tame} Appendix C. Tame estimates.

\section{Introduction}
\label{sec:intro}

A question in control theory for PDEs
regards the persistence of controllability under perturbations. 
In this paper we study the effect of \emph{quasi-linear} perturbations 
(namely nonlinearities containing derivatives of the highest order) 
on the controllability of the KdV equation.
We consider equations of the form
\begin{equation} \label{i1}
u_t + u_{xxx} + \ttf(x,u, u_x, u_{xx}, u_{xxx}) = 0
\end{equation}
on the circle $x \in \T := \R / 2\pi\Z$, with $t \in \R$, 
where $u = u(t,x)$ is real-valued, and $\ttf$ is a given real-valued nonlinear function 
which is at least quadratic around $u=0$.
For solutions of small amplitude, 
\eqref{i1} is a quasi-linear perturbation of the Airy equation $u_t + u_{xxx} = 0$, 
which is the linear part of KdV; then the KdV nonlinear term $u u_x$ 
can be included in $\mN$. 

Motivated by a question, which was posed in \cite{KP}, about the possibility 
of including the dependence on higher derivatives in nonlinear perturbations of KdV, 
equations of the form \eqref{i1} have recently been studied 
in \cite{BBM-Airy, BBM-auto, BBM-mKdV} in the context of KAM theory. 
In this paper we study \eqref{i1} from the point of view of control theory,
proving its exact controllability by means of an internal control, 
in arbitrary time, for sufficiently small data (Theorem \ref{thm:1}).

\medskip

Most of the known results about controllability of quasi-linear PDEs 
deal with first order quasi-linear hyperbolic systems 
of the form $u_t + A(u) u_x = 0$ 
(including quasi-linear wave, shallow water, and Euler equations), 
see for example Li and Zhang \cite{Li-Zhang}, 
Coron \cite{Coron} (chapter 6.2, and see also the many references therein), 
Li and Rao \cite{Li-Rao}, Coron, Glass and Wang \cite{CGW}, 
and recently Alabau-Boussouira, Coron and Olive \cite{ACO}.
Recent results for different kinds of quasi-linear PDEs 
are contained in Alazard, Baldi and Han-Kwan \cite{ABH} 
on the internal controllability of 2D gravity-capillary water waves equations, 
and Alazard \cite{Alaz} on the boundary observability of 2D and 3D (fully nonlinear) 
gravity water waves.
For a general introduction to the theory of control for PDEs see, for example,
Lions \cite{Lions}, Micu and Zuazua \cite{Micu-Zuazua}, Coron \cite{Coron},
while for important results in control for hyperbolic PDEs see, for example, 
Bardos, Lebeau and Rauch \cite{BLR}, Burq and G\'erard \cite{BG}, Burq and Zworski \cite{BZ}.
 
Regarding the KdV equation, the first controllability results are due to 
Zhang \cite{Zhang} and Russell \cite{Russell}. 
Among recent results, we mention the work by Laurent, Rosier and Zhang \cite{LRZ} for large data. 
A beautiful review on the literature on control for KdV can be found in \cite{RZ}. 
For more on KdV, see the rich survey \cite{Guan-Kuksin} by Guan and Kuksin, 
and the many references therein.

\subsection{Main result}

We assume that the nonlinearity $\ttf(x,u,u_x, u_{xx}, u_{xxx})$ 
is at least quadratic around $u=0$, namely the real-valued function 
$\ttf : \T \times \R^4 \to \R$ satisfies 
\begin{equation} \label{i2}
|\ttf(x, z_0, z_1, z_2, z_3)| \leq C |z|^2 
\quad \forall z = (z_0, z_1, z_2, z_3) \in \R^4, \ |z| \leq 1.
\end{equation}
We assume that the dependence of $\ttf$ on $u_{xx}, u_{xxx}$ is Hamiltonian, 
while no structure is required on its dependence on $u, u_x$. More precisely, we assume that 
\begin{equation} \label{i2.1}
\ttf(x,u,u_x, u_{xx}, u_{xxx}) 
= \ttf_1(x,u,u_x, u_{xx}, u_{xxx}) + \ttf_0(x,u,u_x)
\end{equation}
where 
\begin{equation}  \label{i6}
\begin{aligned} 
& \ttf_1(x,u,u_x, u_{xx}, u_{xxx}) 
= \pa_x \{ (\pa_u \mF)(x, u, u_x) \} 
- \pa_{xx} \{ (\pa_{u_x} \mF)(x, u, u_x) \}
\\
& \text{for some function $\mF : \T \times \R^2 \to \R$.}
\end{aligned}
\end{equation}
Note that the case $\ttf = \ttf_1$, $\ttf_0 = 0$ corresponds to the Hamiltonian equation 
$\pa_t u = \pa_x \gr H(u)$ where the Hamiltonian is 
\begin{equation} \label{i3}
H(u) = \frac12 \int_\T u_x^2 \, dx + \int_\T \mF(x,u,u_x) \, dx
\end{equation}
and $\gr$ denotes the $L^2(\T)$-gradient. The unperturbed KdV is the case $\mF = - \frac16 u^3$. 

\medskip

\emph{Notations}. 
For periodic functions $u(x)$, $x \in \T$, we expand 
$u(x) = \sum_{n \in \Z} u_n e^{inx}$, 
and, for $s \in \R$, we consider the standard Sobolev space of periodic functions
\begin{equation} \label{i7}
H^s_x := H^s(\T, \R) := \big\{ u : \T \to \R : 
\| u \|_s < \infty \big\}, \quad 
\| u \|_s^2 := \sum_{n \in \Z} |u_n|^2 \langle n \rangle^{2s},
\end{equation}
where $\langle n \rangle := (1 + n^2)^{\frac12}$. 
We consider the space $C([0,T],H^s_x)$ of functions $u(t,x)$ that are continuous in time 
with values in $H^s_x$. 
We will use the following notation for the standard norm in $C([0,T], H^s_x)$:
\begin{equation} \label{i8}
\| u \|_{T,s} 
:= \| u \|_{C([0,T],H^s_x)} := \sup_{t \in [0,T]} \| u(t) \|_s.
\end{equation}
For continuous functions $a : [0,T] \to \R$, we will denote 
\begin{equation} \label{1609.1}
| a |_T := \sup \{ |a(t)| : t \in [0,T] \}.
\end{equation}

\bigskip

\bigskip

\bigskip

\begin{theorem}[Exact controllability] \label{thm:1}
Let $T>0$, and let $\om \subset \T$ be a nonempty open set. 
There exist positive universal constants $r,s_1$ such that, 
if $\mN$ in \eqref{i1} is of class $C^r$ in its arguments 
and satisfies \eqref{i2}, \eqref{i2.1}, \eqref{i6}, 
then there exists a positive constant $\d_*$ depending on $T,\om,\mN$
with the following property. 

Let $u_{in}, u_{end} \in H^{s_1}(\T,\R)$ 
with 
\[
\| u_{in} \|_{s_1} + \| u_{end} \|_{s_1} \leq \d_*.
\] 
Then there exists a function $f(t,x)$ satisfying
\[
f(t,x) = 0 \quad \text{for all $x \notin \om$, for all $t \in [0,T]$,}
\]
belonging to $C([0,T],H^s_x) 
\cap C^1([0,T],H^{s-3}_x) 
\cap C^2([0,T],H^{s-6}_x)$
for all $s < s_1$,
such that the Cauchy problem
\begin{equation} \label{i9}
\begin{cases}
u_t + u_{xxx} + \ttf(x,u,u_x, u_{xx}, u_{xxx}) = f 
\quad \forall (t,x) \in [0,T] \times \T \\
u(0,x) = u_{in}(x) 
\end{cases}
\end{equation}
has a unique solution $u(t,x)$ belonging to $C([0,T], H^s_x) \cap C^1([0,T], H^{s-3}_x)
\cap C^2([0,T],H^{s-6}_x)$
for all $s < s_1$, which satisfies 
\begin{equation} \label{i10}
u(T,x) = u_{end}(x).
\end{equation}
Moreover, for all $s < s_1$,
\begin{multline} \label{stimetta}
\| u,f \|_{C([0,T],H^s_x)} + \| \pa_t u, \pa_t f \|_{C([0,T],H^{s-3}_x)} 
+ \| \pa_{tt} u, \pa_{tt} f \|_{C([0,T],H^{s-6}_x)} 
\\
\leq C_s (\| u_{in} \|_{s_1} + \| u_{end} \|_{s_1})
\end{multline}
for some $C_s > 0$ depending on $s,T,\om,\mN$. 
\end{theorem} 

\begin{remark} \label{rem:fictitious}
In Theorem \ref{thm:1} there is an arbitrarily small loss of regularity: 
if the initial and final data $u_{in}, u_{end}$ have Sobolev regularity $H^{s_1}_x$, 
then the control $f$ and the solution $u$ are continuous in time with values in $H^s_x$ for all $s<s_1$. Such loss of regularity is in some sense fictitious: it is due to our choice of working 
with standard Sobolev spaces, but it could be avoided by working with the (slightly ``worse-looking'') weak spaces $E_a'$ introduced by H\"ormander in \cite{Olli} 
(see Section \ref{sec:NM}). What we actually prove is that, if the initial and final data 
are in the weak space $(H^{s_1}_x)'$ (i.e.\ the weak version \emph{\`a la} H\"ormander \cite{Olli} 
of the Sobolev space $H^{s_1}_x$), then $f$ and $u$ are continuous in time with values 
in \emph{the same} space $(H^{s_1}_x)'$. 
\end{remark}

\begin{remark}
Our proof of Theorem \ref{thm:1} does not use results of existence and uniqueness for the Cauchy problem \eqref{i9}. On the contrary, our method directly proves local existence and uniqueness for \eqref{i9} (see Theorem \ref{thm:byproduct}). 
This situation occurs quite often in control problems (see Remark 4.12 in \cite{Coron}).
\end{remark}

\subsection{Description of the proof}

It would be natural to try to solve the control problem \eqref{i9}-\eqref{i10} using 
a fixed point argument or the usual implicit function theorem. 
However, this seems to be impossible because of the presence of three derivatives in the nonlinear term. 
A similar difficulty was overcome in \cite{ABH} by using a suitable nonlinear iteration scheme
adapted to quasi-linear problems. 
Such a nonlinear scheme requires to solve a linear control problem with variable coefficients
at each step of the iteration, with no loss of regularity with respect to the coefficients
(i.e., the solution must have the same regularity as the coefficients).
In \cite{ABH} this is achieved by means of para-differential calculus, 
together with linear transformations, Ingham-type inequalities and the Hilbert uniqueness method.

As an alternative method, in this paper we use a Nash-Moser implicit function theorem.
The Nash-Moser approach also demands to solve a linear control problem with variable coefficients,
but it has the advantage of requiring weaker estimates, allowing losses of regularity.
The proof of such weaker estimates is easier to obtain, 
and it does not require the use of powerful techniques like para-differential calculus. 
In this sense our Nash-Moser method is alternative to the method in \cite{ABH}
(for a discussion about pseudo- and para-differential calculus in connection with the Nash-Moser theorem, see, for example, H\"ormander \cite{H-90}, Alinhac and G\'erard \cite{AG}).
On the other hand, the result that we obtain with the Nash-Moser method 
is slightly weaker than the one in \cite{ABH} regarding the regularity of the solution of the nonlinear control problem with respect to the regularity of the data: 
the arbitrarily small loss of regularity in Theorem \ref{thm:1} is discussed in Remark \ref{rem:fictitious}, while Theorem 1.1 of \cite{ABH} 
has no loss of regularity also in the standard Sobolev spaces.

Nash-Moser schemes in control problems for PDEs have been used  
by Beauchard, Coron, Alabau-Boussouira, Olive in \cite{Beau1, BC-JFA-2006, Beau2, ACO}.
A discussion about Nash-Moser as a method to overcome the problem of the loss of derivatives 
in the context of controllability for PDEs can be found in \cite[section 4.2.2]{Coron}.
In \cite{BL} Beauchard and Laurent were able to avoid the use of the Nash-Moser theorem 
in semilinear control problems thanks to some regularizing effect.
We remark that Theorem \ref{thm:1} could also be proved without Nash-Moser 
(for example, by adapting the method of \cite{ABH}). 

\medskip

Now we describe our method in more detail.
Given a nonempty open set $\om \subset \T$, 
we first fix a $C^\infty$ function $\chi_\om(x)$ 
with values in the interval $[0,1]$ which vanishes outside $\om$, 
and takes value $\chi_\om = 1$ on a nonempty open subset of $\om$. 
Thus, given initial and final data $u_{in}, u_{end}$, 
we look for $u,f$ that solve
\begin{equation} \label{ge1}
\begin{cases}
P(u) = \chi_\om f \\
u(0) = u_{in} \\
u(T) = u_{end} 
\end{cases}
\end{equation}
where
\begin{equation} \label{ge2}
P(u) := u_t + u_{xxx} + \ttf(x,u,u_x,u_{xx}, u_{xxx}).
\end{equation}
We define 
\begin{equation} \label{ge3}
\Phi(u,f) := \begin{pmatrix} 
P(u) - \chi_\om f \\
u(0) \\
u(T) \end{pmatrix}
\end{equation}
so that problem \eqref{ge1} is written as 
\[
\Phi(u,f) = (0, u_{in}, u_{end}).
\] 
The crucial assumption to verify in order to apply any Nash-Moser theorem 
is the existence of a right inverse of the linearized operator. 
The linearized operator $\Phi'(u,f)[h,\ph]$ at the point $(u,f)$ in the direction $(h,\ph)$ is
\begin{equation} \label{ge4}
\Phi'(u,f)[h,\ph] := \begin{pmatrix} 
P'(u)[h] - \chi_\om \ph \\
h(0) \\
h(T) \end{pmatrix}.
\end{equation}
Thus we have to prove that, given any $(u,f)$ and any $g := (g_1, g_2, g_3)$ 
in suitable function spaces, there exists $(h,\ph)$ such that 
\begin{equation}  \label{1410}
\Phi'(u,f)[h,\ph] = g.
\end{equation}
Moreover we have to estimate $(h,\ph)$ in terms of $u,f,g$ in a ``tame'' way 
(an estimate is said to be tame when it is linear in the highest norms: see \eqref{tame in NM} and \eqref{2809}).

Problem \eqref{1410} is a linear control problem. 
We observe that the linearized operator $P'(u)[h]$ 
is a differential operator having variable coefficients also at the highest order 
(which is a consequence of linearizing a \emph{quasi-linear} PDE).
Explicitly, it has the form 
\[
P'(u)[h] = \pa_t h + (1 + a_3(t,x)) \pa_{xxx} h 
+ a_2(t,x) \pa_{xx} h 
+ a_1(t,x) \pa_x h 
+ a_0(t,x) h.
\]
We solve \eqref{1410} 
in Theorem \ref{thm:inv}. 
Note that the choice of the function spaces is not given a priori: 
to fix a suitable functional setting is part of the problem. 

Theorem \ref{thm:inv} is proved by adapting a procedure of reduction to constant coefficients developed in \cite{BBM-Airy, BBM-auto}. Such a procedure conjugates $P'(u)$  
to an operator $\mL_5$ (see \eqref{L34}) having constant coefficients up to a bounded remainder.
This conjugation is achieved by means of changes of the space variable, reparametrization of time, multiplication operators, and Fourier multipliers. 
Using Ingham inequality and a perturbation argument we prove 
the observability of $\mL_5$. 
Then we prove the observability of $P'(u)$ exploiting the explicit formulas of the transformations that conjugate $P'(u)$ to $\mL_5$. 
The linear control problem \eqref{1410} is solved in $L^2_x$ 
by the HUM (Hilbert uniqueness method). 
Then further regularity of the solution $(h,\ph)$ of \eqref{1410} 
is proved by adapting an argument used 
by Dehman-Lebeau \cite{Dehman-Lebeau}, Laurent \cite{Laurent}, and \cite{ABH}.

To conclude the proof of Theorem \ref{thm:1} we apply Theorem \ref{thm:NM}, 
which is a modified version of two Nash-Moser implicit function theorems by H\"ormander 
(Theorem 2.2.2 in \cite{Geodesy} and main theorem in \cite{Olli}; 
see also Alinhac-G\'erard \cite{AG}). 
With respect to the abstract theorem in \cite{Olli}, 
our Theorem \ref{thm:NM} assumes slightly stronger hypotheses on the nonlinear operator, 
and it removes two conditions that are assumed in \cite{Olli}, 
which are the compact embeddings in the codomain scale of Banach spaces  
and the continuity of the approximate right inverse of the linearized operator 
with respect to the approximate linearization point. 
This improvement is obtained by adapting the iteration scheme introduced in \cite{Geodesy}.
On the other hand, the Nash-Moser implicit function theorem in \cite{Geodesy} 
holds for H\"older spaces with noninteger indices, and it does not apply to Sobolev spaces 
(in particular, Theorem A.11 of \cite{Geodesy} does not hold % is false 
for Sobolev spaces). 

This method is not confined to KdV,  
and it could be applied to prove controllability of other quasi-linear evolution PDEs.

\medskip

The use of Ingham-type inequalities and HUM is classical in control theory
(see, for example, \cite{H Lacun, Micu-Zuazua, Komornik-Loreti, Kahane} for Ingham 
and \cite{Lions, Micu-Zuazua, Coron, Komornik} for HUM).
As mentioned above, the Nash-Moser theorem has also been used in control theory 
(see, for example, \cite{Beau1, BC-JFA-2006, Beau2, ACO}). 
It was first introduced by Nash \cite{Nash}, then several 
refinements were developed afterwards,  
see for example Moser \cite{Moser}, Zehnder \cite{Zehnder},
Hamilton \cite{Hamilton}, Gromov \cite{Gromov}, H\"ormander \cite{Geodesy, Olli, H-90}, 
and, recently, Berti, Bolle, Corsi and Procesi \cite{BBPro, BCPro}, 
Ekeland and S\'er\'e \cite{Ekeland, ES}. 
For our problem, H\"ormander's versions \cite{Geodesy, Olli} 
seem to be the best ones concerning the loss of regularity of the solution 
with respect to the regularity of the data (see also Remark \ref{rem:fictitious}).
As already said, the theorems in \cite{Geodesy, Olli} cannot be applied directly, 
but they can be adapted to our goal. 
This is the content of Section \ref{sec:NM}.

\subsection{Byproduct: a local existence and uniqueness result}
As a byproduct, with the same technique and no extra work, 
we have the following existence and uniqueness theorem for the Cauchy problem 
of the quasi-linear PDE \eqref{i1}.

\begin{theorem}[Local existence and uniqueness] \label{thm:byproduct}
There exist positive universal constants $r,s_0$ such that, 
if $\mN$ in \eqref{i1} is of class $C^r$ in its arguments 
and satisfies \eqref{i2}, \eqref{i2.1}, \eqref{i6}, 
then the following property holds.
For all $T > 0$ there exists $\d_* > 0$ such that
for all $u_{in} \in H^{s_0}_x$, $f \in C([0,T], H^{s_0}_x) \cap C^1([0,T], H^{s_0-6}_x)$ 
(possibly $f=0$) satisfying
\begin{equation} \label{i12}
\| u_{in} \|_{s_0} + \| f \|_{T,s_0} + \| \pa_t f \|_{T,s_0-6} \leq \d_* \,,
\end{equation}
the Cauchy problem 
\begin{equation} \label{i11}
\begin{cases}
u_t + u_{xxx} + \ttf (x,u,u_x, u_{xx}, u_{xxx}) = f, 
\qquad (t,x) \in [0,T] \times \T \\
u(0,x) = u_{in}(x) 
\end{cases}
\end{equation}
has one and only one solution $u \in C([0,T], H^s_x) \cap C^1([0,T], H^{s-3}_x) \cap C^2([0,T], H^{s-6}_x)$ for all $s< s_0$.
Moreover, for all $s < s_0$,
\begin{multline} \label{stimetta bis}
\| u \|_{C([0,T],H^s_x)} + \| \pa_t u \|_{C([0,T],H^{s-3}_x)} 
+ \| \pa_{tt} u \|_{C([0,T],H^{s-6}_x)} 
\\
\leq C_s \Big( \| u_{in} \|_{s_0} + \| f \|_{C([0,T],H^{s_0}_x)}
+ \| \pa_t f \|_{C( [0,T],H^{s_0-6}_x)} \Big)
\end{multline}
for some $C_s > 0$ depending on $s,T,\mN$. 
\end{theorem} 

\begin{remark}
Theorem \ref{thm:byproduct} is not sharp: 
we expect that better results for the Cauchy problem \eqref{i11} can be proved 
by using a para-differential approach. 
\end{remark}

\begin{remark} The loss of regularity in Theorem \ref{thm:byproduct} 
is of the same type as the one in Theorem \ref{thm:1}, 
see the discussion in Remark \ref{rem:fictitious}.
\end{remark}

\subsection{Organization of the paper}

In Section \ref{sec:regu} we describe the transformations that conjugate the linearized operator $P'(u)$ to constant coefficients up to a bounded remainder, and we give quantitative estimates on these transformations. 
In Section \ref{sec:obs} we exploit these results to prove the observability of $P'(u)$.
In Section \ref{section:con} we use observability to solve the linear control problem \eqref{1410} via HUM (Theorem \ref{thm:inv}) and we fix suitable function spaces \eqref{def Es}-\eqref{def Fs}.
In Section \ref{sec:proof} we prove Theorems \ref{thm:1} and \ref{thm:byproduct} 
by applying Theorem \ref{thm:NM}.
In Section \ref{sec:WP} we prove well-posedness with tame estimates for all the linear operators 
involved in the reduction procedure.
These well-posedness results are used many times along the Sections \ref{sec:obs}, \ref{section:con}, \ref{sec:proof}. 
In Section \ref{sec:NM} we prove Nash-Moser Theorem \ref{thm:NM}. 
In Section \ref{sec:tame} we recall standard tame estimates that are used in the rest of the paper.

\bigskip

\begin{small}
\noindent
\textbf{Acknowledgements}.
We thank Thomas Alazard and Daniel Han-Kwan for useful comments and discussions.
We thank the anonymous referee for pointing out a defect, 
which now has been corrected, 
and for other useful suggestions. 

This research was carried out in the frame of Programme \textsc{star}, 
financially supported by UniNA and Compagnia di San Paolo.
This research was supported by the European Research Council under FP7 (\textsc{erc} Project 306414),
by \textsc{prin} 2012 ``Variational and perturbative aspects of nonlinear differential problems'',
by \textsc{indam-gnampa} Research Project 2015 
``Analisi e controllo di equazioni a derivate parziali nonlineari'', 
by \textsc{gdre conedp} ``Control of Partial Differential Equations'' issued by \textsc{cnrs}, 
\textsc{indam} and Universit\'e de Provence. 
\end{small}

\section{Reduction of the linearized operator to constant coefficients} 
\label{sec:regu}

In this section we consider some changes of variables that conjugate the linearized operator 
to constant coefficients up to a bounded remainder. 
This reduction procedure closely follows the analysis in \cite{BBM-Airy} and \cite{BBM-auto}, 
with some adaptations. 

The linearized operator $P'(u)$ is
\begin{equation} \label{L1}
P'(u)[h] = \pa_t h + (1 + a_3) \pa_{xxx} h + a_2  \pa_{xx} h + a_1 \pa_{x} h + a_0 h,
\end{equation}
where the coefficients $a_i = a_i(t,x)$, $i=0,\ldots,3$ 
are real-valued functions of $(t,x) \in [0,T] \times \T$, 
depending on $u$ by 
\begin{equation} \label{ai}
a_i = a_i(u) := (\pa_{z_i} \ttf)(x,u,u_x, u_{xx}, u_{xxx}), \quad i = 0,\ldots,3
\end{equation}
(recall the notation $\ttf = \ttf(x,z_0,z_1,z_2,z_3)$).
Note that $a_2 = 2 \pa_x a_3$ because of the Hamiltonian structure of the component $\ttf_1$ 
of the nonlinearity (see \eqref{i2.1}-\eqref{i6}).
%The coefficients $a_i$ satisfy the following tame estimates. 
%(as a consequence of the classical tame estimates for composition of functions, see Lemma \ref{lemma:tame cf}).

\begin{lemma}  \label{lemma:stime ai} 
Let $ \ttf \in C^r(\T \times \R^4, \R)$ satisfying \eqref{i2}. 
For all $1 \leq s \leq r - 3$, and for all $u \in C^2 ([0,T], H^{s+3}_x)$ such that $\| u, \pa_t u, \pa_{tt} u \|_{T,4} \leq 1$, the coefficients $a_i(u)$ satisfy 
\begin{equation} \label{L2}
\| a_i(u), \pa_t a_i(u), \pa_{tt} a_i(u) \|_{T,s} 
\leq C \| u, \pa_t u, \pa_{tt} u \|_{T,s+3}, \qquad i=0,1,2,3.
\end{equation}
%where $C_s$ depends on $s$.
\end{lemma}

\begin{proof}
Apply standard tame estimates for composition of functions, see Lemma \ref{lemma:tame cf}.
\end{proof}

Now we apply the reduction procedure 
to any linear operator of the form \eqref{L1} where 
\begin{equation}\label{a2a3}
a_2 (t,x) = c \pa_x a_3 (t,x)
\end{equation}
for some constant $c \in \R$ (note that $P'(u)$ has $c=2$ because of the Hamiltonian structure of $\ttf_1$).
Regarding the loss of regularity with respect to the space variable $x$,
the estimates in the sequel will be not sharp. 
In the whole section we consider $T > 0$ fixed, and, unless otherwise specified, 
all the constants may depend on $T$.

\begin{remark}\label{come fosse Ham}
Given a linear operator $\mL_0$ of the form \eqref{L1}, define the operator $\mL_0^*$ as
\begin{equation}\label{Lstar}
\mL_0^* h := - \pa_t h - \pa_{xxx} \{ (1 + a_3) h \} + \pa_{xx} (a_2 h) - \pa_x (a_1 h) + a_0 h\ .
\end{equation}
Note that $- \mL_0^*$ is still an operator of the form \eqref{L1}, namely
\begin{equation}\label{Lstar.2}
-\mL_0^* = \pa_t + (1+a_3^*) \pa_{xxx} + a_2^* \pa_{xx} + a_1^* \pa_x + a_0^*
\end{equation}
with
\begin{alignat}{2} \label{Lstar.3}
a_3^* & := a_3, 
\qquad & 
a_2^* & := 3 (a_3)_x - a_2, 
\\ 
a_1^* & := 3 (a_3)_{xx} - 2 (a_2)_x +a_1, 
\qquad & 
a_0^* & := (a_3)_{xxx} - (a_2)_{xx} + (a_1)_x - a_0. \notag
\end{alignat}
It follows from \eqref{Lstar.2}, \eqref{Lstar.3} that if $\mL_0$ satisfies \eqref{a2a3}, then also $-\mL_0^*$ satisfies \eqref{a2a3} (with a different constant), namely $a_2^* = (3 - c) \pa_x a_3^*$. In particular, if $\mL_0$ satisfies \eqref{a2a3} with $c=2$ (which is the case if $\mL_0 = P'(u)$), then $-\mL_0^*$ satisfies \eqref{a2a3} with $c=1$.
\end{remark}

\subsection{Step 1. Change of the space variable} \label{step-1}

We consider a $t$-dependent family of diffeomorphisms of the circle $\T$ of the form 
\begin{equation}\label{L3}
y = x + \b(t,x), 
\end{equation}
where $\b$ is a real-valued function, $2\pi$ periodic in $x$, defined for $t \in [0,T]$, 
with $|\b_x(t,x)| \leq 1/2$ for all $(t,x) \in [0,T] \times \T$. 
We define the linear operator 
\begin{equation}\label{L4}
({\cal A}h)(t,x) := h(t, x + \b(t,x)). 
\end{equation}
The operator $\mA$ is invertible, with inverse $\mA^{-1}$, transpose $\mA^T$ 
(transpose with respect to the usual $L^2_x$-scalar product) and inverse transpose $\mA^{-T}$ given by
\begin{equation}\label{L5}
\begin{aligned}
({\cal A}^{-1} v)(t,y) & =  v(t,y + \tilde \b (t,y)),
\quad 
(\mA^T v)(t,y) = (1 + \tilde \b_y(t,y)) \, v(t,y + \tilde \b (t,y)),\\
({\cal A}^{-T}h)(t,x) & = (1+ \b_x (t,x)) \, h(t, x + \b(t,x))
\end{aligned}
\end{equation}
where $y \mapsto y + \tilde \b(t,y)$ is the inverse diffeomorphism of \eqref{L3}, namely
\begin{equation} \label{L6}
x = y + \tilde \b(t,y) \quad \Longleftrightarrow \quad y = x + \b(t,x).  
\end{equation}

Given the operator
\begin{equation}\label{L0}
\mL_0 :=  \pa_t + (1 + a_3 (t,x)) \pa_{xxx} + a_2 (t,x)  \pa_{xx} + a_1 (t,x) \pa_{x} + a_0 (t,x) \ ,
\end{equation}
with $a_2 (t,x) = c \pa_x a_3 (t,x)$ we calculate the conjugate ${\cal A}^{-1} \mL_0 {\cal A}$. 
The conjugate $ {\cal A} \inv a {\cal A} $ of any multiplication operator 
$a : h(t,x) \mapsto a(t,x) h(t,x)$ is the multiplication operator 
$(\mA^{-1} a)$ that maps $v(t,y) \mapsto (\mA^{-1} a)(t,y) \, v(t,y)$.
By conjugation, the differential operators become 
\begin{align*}
{\cal A}\inv \pa_t {\cal A} = \pa_t + (\mA\inv \b_t) \pa_y 
\qquad 
{\cal A}\inv \pa_x {\cal A} = \{ {\cal A}\inv (1 + \b_x) \} \, \pa_y 
\end{align*}
then $\mA^{-1} \pa_{xx} \mA = (\mA^{-1} \pa_x \mA)(\mA^{-1} \pa_x \mA)$, 
and similarly for the conjugate of $\pa_{xxx}$. 
We calculate
\begin{equation} \label{L7}
\mL_1 
:= {\cal A}^{-1} \mL_0 {\cal A} 
= \pa_t + a_4(t,y) \pa_{yyy} 
+ a_5(t,y) \pa_{yy} 
+ a_6(t,y) \pa_{y} 
+ a_7(t,y)
\end{equation} 
where 
\begin{align}
\notag
a_4 & = \mA\inv \{ (1+a_3) (1+\b_x)^3 \}, 
\qquad 
a_5 = \mA\inv \{ a_2 (1 + \b_x)^2 + 3 (1 + a_3) \b_{xx} (1+\b_x) \},
\\
a_6 & = \mA\inv \{ \b_t + (1+a_3) \b_{xxx} + a_2 \b_{xx} + a_1 (1+\b_x) \}, \qquad a_7 = \mA\inv a_0.
\label{L8} 
\end{align}

We look for $\b(t,x)$ such that the coefficient $a_4(t,y)$ of the highest order derivative $\partial_{yyy}$ in \eqref{L7} does not depend on $y$, namely 
$a_4(t,y) = b(t)$ for some function $b(t)$ of $t$ only. 
This is equivalent to 
\begin{equation}\label{L11}
\big( 1 + a_3(t,x) \big) \big( 1 + \b_x(t,x) \big)^3 = b(t),
\end{equation}
namely 
\begin{equation} \label{L12}
\beta_{x} = \rho_0, \qquad 
\rho_0(t,x) := b(t)^{1/3} \big( 1 + a_3(t,x) \big)^{-1/3} - 1. 
\end{equation} 
The equation \eqref{L12} has a solution $\b$, periodic in $ x $, if and only if
$ \int_{\T}{\rho_0(t,x) \, dx} = 0$ for all $t$. 
This condition uniquely determines 
\begin{equation} \label{L13}
b(t) = \left( \frac{1}{2\pi}\int_{\T} \big( 1 + a_3(t,x) \big)^{-\frac13} \, dx \right)^{-3}.
\end{equation}
Then we fix the solution (with zero average) of \eqref{L12}, 
\begin{equation} \label{L14}
\beta(t,x) := \, (\pa_x\inv \rho_0)(t,x), 
\end{equation}
where $\pa_x\inv h$ is the primitive of $h$ with zero average in $x$ (defined in Fourier).
We have conjugated $\mL_0$ to 
\begin{equation}  \label{L15}
\mL_1 = {\cal A}^{-1} \mL_0 {\cal A}  
= \pa_t + a_4(t) \pa_{yyy} + a_5 (t,y) \pa_{yy} + a_6(t,y) \pa_y + a_7(t,y),
\end{equation} 
where $a_4(t) := b(t)$ is defined in \eqref{L13}. 

We prove here some bounds that will be used later.

\begin{lemma} \label{lemma:mA}
There exist positive constants $\s, \d_*$ with the following properties.
Let $s \geq 0$, and let $a_3(t,x), a_2(t,x),a_1(t,x), a_0(t,x)$ be four functions with 
$a_2 = c \pa_x a_3$ for some $c\in\R$. Moreover, assume $\pa_{tt} a_3, \pa_t a_3, a_3, \pa_t a_1, a_1, a_0 \in C([0,T],H^{s+\s}_x)$. 
Let 
\begin{equation}\label{m.1}
\d(\mu) := \| \pa_{tt} a_3, \pa_t a_3, a_3, \pa_t a_1, a_1, a_0 \|_{T,\mu + \s} 
\quad \forall \mu \in [0,s].
\end{equation}
If $\d(0) \leq \d_*$, then the operator $\mA$ 
defined in \eqref{L4}, \eqref{L14}, \eqref{L12}, \eqref{L13} 
belongs to $C([0,T], \mL(H^\mu_x))$ for all $\mu \in [0,s]$ 
and satisfies 
\begin{equation} \label{m.2}
\| \mA h \|_{T,\mu} 
\leq C_\mu \big( \| h \|_{T,\mu} + \d(\mu) \| h \|_{T,0} \big)
\quad \forall h \in C([0,T],H^\mu_x),
\end{equation}
for some positive $C_\mu$ depending on $\mu$.
The inverse operator $\mA^{-1}$, 
the transpose $\mA^T$ and the inverse transpose $\mA^{-T}$ 
all satisfy the same estimate \eqref{m.2} as $\mA$.

The functions $a_4(t) = b(t)$, $a_5(t,y)$, $a_6(t,y)$, $a_7(t,y)$, $\b(t,x)$, $\tilde \b(t,y)$ 
defined in \eqref{L13}, \eqref{L12}, \eqref{L14}, \eqref{L8}, \eqref{L6} 
belong to $C([0,T], H^\mu_x)$ for all $\mu \in [0,s]$ and satisfy
\begin{equation} \label{m.3}
\| \b, \tilde\b, a_5, \pa_t a_5, a_6, \pa_t a_6, a_7 \|_{T,\mu} + |a_4-1, a_4'|_T
\leq C_\mu \d(\mu)\ .
\end{equation}
Finally, the coefficient $a_5 (t, y)$ satisfies
\begin{equation}\label{m.4}
\int_\T a_5 (t ,y) \, dy = 0 \qquad \forall t \in [0,T]\ .
\end{equation}

\end{lemma}

\begin{proof} 
The proof of \eqref{m.2} and \eqref{m.3} is a straightforward application of the standard tame estimates for products, composition of functions and changes of variable, 
see section \ref{sec:tame}.

To prove \eqref{m.4}, we use the definition of $b(t)$ in \eqref{L13}, the equality $a_2= c \pa_x a_3$, and the change of variables \eqref{L6}, and we compute
\begin{align*}
\int_\T a_5 (t, y) \, dy 
& = \int_\T [a_2 (1 + \b_x)^2 + 3 (1 + a_3) \b_{xx} (1+ \b_x) ] (1 + \b_x) \, dx \\
& =  b(t) \left\{ c \int_\T \frac{ \pa_x a_3 (t, x)}{1 + a_3 (t, x)} \, dx + 3 \int_\T \frac{\b_{xx} (t, x)}{1 + \b_x (t, x)} \, dx \right\} \\
& = b(t) \left\{ c \int_\T \pa_x \log (1 + a_3 (t, x)) \, dx + 3 \int_\T  \pa_x \log (1 + \b_x (t, x)) \, dx \right\} = 0. \ \qedhere
\end{align*}
\end{proof}

\subsection{Step 2. Time reparametrization} \label{step-2} 

The goal of this section is to obtain a constant coefficient instead of $a_4(t)$. 
We consider a diffeomorphism $\psi:[0,T] \to [0,T]$ 
which gives the change of the time variable 
\begin{equation} \label{L16}
\psi(t) = \tau \quad \Leftrightarrow \quad t = \psi^{-1}(\t),
\end{equation}
with $\psi(0) =0$ and $\psi(T) = T$. 
We define 
\begin{equation}\label{L17}
(\mB h)(t,y):= h(\psi(t), y), \quad 
(\mB^{-1} v)(\t,y) := v(\psi^{-1}(\t), y).
\end{equation}
By conjugation, the differential operators become 
\begin{equation}  \label{L18}
\mB^{-1} \pa_t \mB = \rho(\t) \pa_\t, \quad 
\mB^{-1} \pa_y \mB = \pa_y, \quad 
\rho := \mB^{-1}(\psi'),
\end{equation}
and therefore \eqref{L15} is conjugated to
\begin{equation}\label{L19}
\mB\inv \mL_1 \mB = \rho \,  \pa_\t  
+ (\mB\inv a_4) \pa_{yyy} 
+ (\mB\inv a_5) \pa_{yy} 
+ (\mB\inv a_6) \pa_{y} 
+ (\mB\inv a_7).
\end{equation}
We look for $\psi$ such that the (variable) coefficients of the highest order derivatives 
($\pa_\t $ and $\pa_{yyy}$) are proportional, namely
\begin{equation} \label{L20}
(\mB\inv a_4)(\t) 
= m \rho(\t) 
= m (\mB\inv( \psi')) (\t)
\end{equation}
for some constant $m \in \R$. 
Since $\mB$ is invertible, this is equivalent to requiring that
\begin{equation} \label{L21}
a_4(t) = m \psi'(t).
\end{equation}
Integrating on $[0,T]$ determines the value of the constant $m$, 
and then we fix $\psi$:
\begin{equation}  \label{L22}
m := \frac{1}{T} \, \int_0^T a_4(t) \, dt, \quad 
\psi(t) := \frac{1}{m} \int_0^t a_4(s) \, ds.
\end{equation} 
With this choice of $\psi$ we get 
\begin{equation}\label{L23}
\mB\inv \mL_1 \mB = \rho \, \mL_2, 
\qquad 
\mL_2 := \pa_\t + m \pa_{yyy} 
+ a_8 (\t,y) \, \pa_{yy} 
+ a_9(\t,y) \, \pa_{y} 
+ a_{10} (\t,y),
\end{equation}
where
\begin{align}\label{L24}
a_8(\t,y) & := \frac{1}{\rho(\t)}\,(\mB\inv a_5)(\t,y), \qquad 
a_9(\t,y) := \frac{1}{\rho(\t)}\,(\mB\inv a_6)(\t,y), 
\\
a_{10}(\t,y) & := \frac{1}{\rho(\t)}\,(\mB\inv a_7)(\t,y). \notag
\end{align}
Note that for all $\t \in [0,T]$ one has
\begin{equation}\label{L24.1}
\int_\T a_8 (\t, y) \, dy = \frac{1}{(\mB\inv \psi')(\t)} \int_\T (\mB\inv a_5) (\t, y) \, dy
= \frac{1}{\psi' (t)} \int_\T a_5 (t, y) \, dy = 0\ .
\end{equation}

By straightforward calculations, we prove the following lemma.

\begin{lemma}
\label{lemma:mB}
There exists $\d_* > 0$ with the following properties.
Let $a_4 \in C([0,T], \R)$ with $|a _4(t) - 1| \leq \d_*$ for all $t \in [0,T]$.
Then the operator $\mB$ defined in \eqref{L17}, \eqref{L22} 
is an invertible isometry of $C([0,T],H^s_x)$ for all $s \geq 0$, namely
\begin{equation} \label{m.9}
\| \mB h \|_{T,s} = \| h \|_{T,s} \quad \forall h \in C([0,T], H^s_x), \quad s \geq 0.
\end{equation}

Moreover there exists a positive constant $\s$ with the following property.
Let $a_4 \in C^1([0,T],\R)$, with 
$|a_4(t) - 1| \leq \d_*$ and $|a_4'(t)| \leq 1$ for all $t \in [0,T]$.
Let $s \geq 0$, and $a_5, \pa_t a_5, a_6, \pa_t a_6, a_7 \in C([0,T],H^s_x)$ with $\int_\T a_5 (t,y) \, dy = 0$ for all $t \in [0,T]$. 
Then the functions $a_8(t,x)$, $a_9(t,x)$, $a_{10}(t,x)$, $\psi(t)$, $\rho(t)$ and the constant $m$  
defined in \eqref{L24}, \eqref{L22}, \eqref{L18} satisfy
\begin{equation} \label{m.11}
|m-1| + |\psi' -1, \rho -1|_T 
+ \| a_8, \pa_\t a_8, a_9, \pa_\t a_9, a_{10} \|_{T,s} 
\leq C \| a_5, \pa_t a_5, a_6, \pa_t a_6, a_7 \|_{T,s} 
\end{equation}
where $C$ is independent of $s$. Moreover one has
\begin{equation}\label{m.12}
\int_\T a_8 (\t ,y) \, dy = 0 \qquad \forall \t \in [0,T]\ .
\end{equation}
\end{lemma}

\subsection{Step 3. Multiplication}\label{step-3}

In this section we eliminate the term $a_8 (\t, y) \pa_{yy}$ from the operator $\mL_2$ defined in \eqref{L23}.
To this end, we consider the multiplication operator $\mM$ defined as
\begin{equation}\label{M1}
\mM h (\t,y) := q(\t,y) h(\t,y)
\end{equation}
with $q:[0,T]\times \T \to \R$. We compute
\begin{equation}\label{M2}
\mM\inv \mL_2 \mM = \pa_\t + m \pa_{yyy} + a_{11} (\t,y) \pa_{yy} + a_{12} (\t,y) \pa_y + a_{13} (\t,y)
\end{equation}
with
\begin{equation}\label{M3}
a_{11}:= a_8 + \frac{3m q_y}{q}, \qquad
a_{12}:= a_9 + \frac{2 a_8 q_y + 3m q_{yy}}{q}, \qquad
a_{13}:= \frac{\mL_2 q}{q}.
\end{equation}
We want to choose $q$ such that $a_{11}=0$, which is equivalent to
\begin{equation}\label{M4}
3 m q_y + a_8 q = 0\ .
\end{equation}
Thanks to \eqref{m.12}, equation \eqref{M4} admits the space-periodic solution
\begin{equation}\label{M5}
q(\t,y):=\exp\Big\{ -\frac{1}{3m} ( \pa_y\inv a_8 ) (\t, y) \Big\}\ .
\end{equation}
As a consequence, we get
\begin{equation}\label{M7}
\mL_3 := \mM\inv \mL_2 \mM = \pa_\t + m \pa_{yyy} + a_{12} (\t,y) \pa_y + a_{13} (\t,y)\ .
\end{equation}

The proof of the following lemma is straightforward.

\begin{lemma}
\label{lemma:mM}
Let $s\geq 0$ and let $a_8\in C([0,T], H^s_x)$ with $\int_\T a_8 (\t, y) \, dy = 0$ for all $\t \in [0,T]$. Then for all $\mu\in[0,s]$, the operator $\mM$ defined in \eqref{M1}, \eqref{M5} and its inverse $\mM\inv$ belong to $C([0,T], \mL (H^\mu_x))$. Note that $\mM=\mM^T$.

Furthermore, there exist two positive constants $\d_*,\s$ with the following properties.
Assume that $a_8, \pa_t a_8, a_9, \pa_t a_9, a_{10}\in C([0,T], H^{s+\s}_x)$ and let
\begin{equation}\label{M8}
\d(\mu) := \| a_8, \pa_t a_8, a_9, \pa_t a_9, a_{10} \|_{T,\mu + \s}\ .
\end{equation}
Then if $\d(0) \leq \d_*$, for all $\mu\in[0, s]$ the operator $\mM$ and its inverse $\mM\inv$ satisfy
\begin{equation}\label{M9}
\| \mM^{\pm 1} h \|_{T,\mu} \leq
C_\mu \big( \| h \|_{T,\mu} + \d(\mu) \| h \|_{T,0} \big)
\quad \forall h \in C([0,T],H^\mu_x),
\end{equation}
for some positive $C_\mu$ depending on $\mu$.
Moreover, the functions $a_{12} (\t,y), a_{13} (\t, y), q(\t, y)$ defined in \eqref{M3}, \eqref{M5} satisfy
\begin{equation} \label{M10}
\| q-1, a_{12}, \pa_t a_{12}, a_{13} \|_{T,\mu} 
\leq C_\mu \d(\mu)\ .
\end{equation}
\end{lemma}

\subsection{Step 4. Translation of the space variable}\label{step-4}

We consider the change of the space variable $z = y + p(\t)$
and the operators 
\begin{equation}\label{L25}
{\cal T} h(\t,y) := h(\t, y + p(\t)), \quad 
{\cal T}\inv v(\t,z) := v(\t, z - p(\t))
\end{equation}
where $p$ is a function $p : [0,T] \to \R$. 
The differential operators become 
${\cal T}\inv \pa_y {\cal T} = \pa_z$ 
and ${\cal T}\inv \pa_\t  {\cal T}$ $= \pa_\t  + \{ \pa_\t  p(\th) \} \, \pa_z$. 
This is a special, simple case of the transformation $\mA$ of section \ref{step-1}.
Thus
\begin{equation}\label{L26}
\mL_4 := {\cal T}\inv \mL_3 {\cal T} 
= \pa_\t  + m \pa_{zzz} + a_{14}(\t,z) \pa_z + a_{15}(\t,z)
\end{equation}
where 
\begin{equation}\label{L27}
a_{14}(\t,z) := p'(\t) + ({\cal T}\inv a_{12}) (\t,z), 
\quad 
a_{15}(\t,z) := ({\cal T} \inv a_{13})(\t,z).
\end{equation}
Now we look for $p(\t)$ such that $a_{14}$ has zero space average. 
We fix 
\begin{equation} \label{L28}
p(\t) := - \frac{1}{2\p} \, \int_0^\t \int_\T a_{12} (s,y) \, dy ds.
\end{equation}
With this choice of $p$, after renaming the space-time variables $z=x$ and $\t=t$, we have 
\begin{equation} \label{L29}
\mL_4 = \pa_t + m \pa_{xxx} + a_{14}(t,x) \pa_x + a_{15}(t,x),
\qquad 
\int_{\T} a_{14}(t,x) \, dx = 0 \quad \forall t \in [0,T]. 
\end{equation}

With direct calculations we prove the following estimates.

\begin{lemma}
\label{lemma:mT}
Let $a_{12} \in C([0,T],L^2_x)$. 
Then the operator $\mT$ defined in \eqref{L25}, \eqref{L28} 
belongs to $C([0,T], \mL(H^s_x))$ for all $s \in [0,+\infty)$.  
In fact $\mT$ is an isometry, namely
\begin{equation} \label{m.21}
\| \mT h \|_{T,s} 
= \| h \|_{T,s} 
\quad \forall h \in C([0,T],H^s_x). 
\end{equation}
Moreover, $\mT$ is invertible and its transpose is $\mT^T = \mT^{-1}$. 

Let $s \geq 0$, and let $a_{12}, \pa_t a_{12}, a_{13} \in C([0,T],H^{s+1}_x)$ 
with $\| a_{12} \|_{T,0} \leq 1$. 
Then the functions $a_{14}, a_{15}, p$ 
defined in \eqref{L27}, \eqref{L28} satisfy
\begin{equation} \label{m.13}
\sup_{t \in [0,T]} |p(t)| + \| a_{14}, \pa_t a_{14}, a_{15} \|_{T,s} 
\leq C \| a_{12}, \pa_t a_{12}, a_{13} \|_{T,s+1}
\end{equation}
where $C$ is independent of $s$. 
\end{lemma}

\subsection{Step 5. Elimination of the order one}\label{step-5}

The goal of this section is to eliminate the term $a_{14}(t,x) \pa_x$.
Consider an operator $\mS$ of the form
\begin{equation}\label{L30}
\mS h := h + \g(t,x) \pa_x^{-1} h
\end{equation}
where $\g(t,x)$ is a function to be determined. 
Note that $\partial_x^{-1} \partial_x =  \partial_x \partial_x^{-1} = \pi_0$
where $\pi_0 h := h - \frac{1}{2\pi} \int_\T h \, dx$. 
We directly calculate 
\begin{equation} \label{L31}
\mL_4 \mS - \mS (\pa_t + m \partial_{xxx}) 
= a_{16} \partial_{x} + a_{17} + a_{18} \partial_{x}^{-1} 
\end{equation}
where 
\begin{equation} \label{L32} 
\begin{aligned}
a_{16} & := 3 m \g_x + a_{14}, \qquad \ \ 
a_{17} := a_{15} + ( 3 m \g_{xx} + a_{14} \g ) \pi_0, \\
a_{18} & := \g_t + m \g_{xxx} + a_{14} \g_x + a_{15} \g.
\end{aligned}
\end{equation}
We fix $\g$ as 
\begin{equation}\label{L33}
\g := - \frac{1}{3 m} \, \pa_x^{-1} a_{14},
\end{equation}
so that $a_{16} = 0$. 
By the following Lemma \ref{lemma:mS}, ${\cal S}$ is invertible, and we obtain
\begin{equation}\label{L34}
\mL_5 := \mS^{-1} \mL_4 \mS 
= \pa_t + m \pa_{xxx} + \mR, \qquad 
\mR := \mS^{-1} ( a_{17} + a_{18} \partial_{x}^{-1} ).
\end{equation}

\begin{lemma}
\label{lemma:mS}
There exist positive constants $\s, \d_*$ with the following properties.
Let $s \geq 0$, let $a_{14}, a_{15}$ be two functions with 
$a_{14}, \pa_t a_{14}, a_{15} \in C([0,T],H^{s+\s}_x)$ and 
$\int_\T a_{14} (t,x) \, dx = 0$. 
Let 
\begin{equation}\label{m.14}
\d(\mu) := \| a_{14}, \pa_t a_{14}, a_{15} \|_{T,\mu + \s} 
\quad \forall \mu \in [0,s].
\end{equation}
If $\d(0) \leq \d_*$, then the operator $\mS$ defined in \eqref{L30}, \eqref{L33} 
belongs to $C([0,T], \mL(H^\mu_x))$ for all $\mu \in [0,s]$ 
and satisfies 
\begin{equation} \label{m.16}
\| \mS h \|_{T,\mu} 
\leq C_\mu \big( \| h \|_{T,\mu} + \d(\mu) \| h \|_{T,0} \big)
\quad \forall h \in C([0,T],H^\mu_x),
\end{equation}
for some positive $C_\mu$ depending on $\mu$.
The operator $\mS$ is invertible, and its inverse $\mS^{-1}$, 
its transpose $\mS^T$ and its inverse transpose $\mS^{-T}$ 
all satisfy the same estimate \eqref{m.16} as $\mS$.

The operator $\mR$ defined in \eqref{L34} belongs to 
$C([0,T], \mL(H^\mu_x))$ for all $\mu \in [0,s]$ and it satisfies
\begin{equation} \label{m.15}
\| \mR h \|_{T,\mu} 
\leq C_\mu \big( \d(0) \| h \|_{T,\mu} + \d(\mu) \| h \|_{T,0} \big)
\quad \forall h \in C([0,T],H^\mu_x).
\end{equation}
The transpose $\mR^T$ belongs to $C([0,T], \mL(H^\mu_x))$
and satisfies the same estimate \eqref{m.15} as $\mR$. 
\end{lemma}

\begin{proof}
Estimate $\| \g \pa_x^{-1} h \|_{T,\mu}$ by the usual tame estimates for the product of two functions (Lemma \ref{lemma:tame basic}), 
then use Neumann series in its tame version. 
\end{proof}

\section{Observability}
\label{sec:obs}

In this section we prove the observability of linear operators of the form \eqref{L0}. Such observability property will be used in Section \ref{section:con} in order to prove controllability of the linearized problem.
We split the proof into several simple lemmas, 
starting with a direct consequence of Ingham inequality. Since we actually need observability of a Cauchy problem flowing backwards in time (see Lemma \ref{lemma:LM}) with datum at time $T$, we will accordingly state our lemmas. 

\begin{lemma}
[Ingham inequality for $\pa_t + m \pa_{xxx}$]
\label{lemma:Ingham}
For every $T>0$ there exists a positive constant $C_1(T)$ such that, 
for all $(w_n)_{n \in \xZ} \in \ell^2(\xZ , \xC)$, all $m \geq 1/2$,  
$$
\int_{0}^T \Big| \sum_{n\in\xZ} w_n e^{i m n^3 t} \Big|^2\, dt
\geq C_1(T) \sum_{n\in\xZ}|w_n|^2.
$$
\end{lemma} 

\begin{proof} 
See, for example, Theorem 4.3 in Section 4.1 of \cite{Micu-Zuazua}. 
The fact that the constant $C_1(T)$ does not depend on $m$ is obtained by 
closely following the proof in \cite{Micu-Zuazua}, 
and taking into account the lower bound 
for the distance between two different eigenvalues 
$|m n^3 - m k^3| \geq m \geq \frac12$, for all $n, k \in \Z$, $n \neq k$. 
\end{proof}

The following observability result is classical 
(see, e.g., \cite{Russell-Zhang} for a closely related result); 
for completeness, we also give here its proof.

\begin{lemma}[Observability for $\pa_t + m \pa_{xxx}$]
\label{lemma:Obs1} 
Let $T > 0$, and let $\om \subset \T$ be an open set.
Let $v_T \in L^2(\T)$, $m \geq 1/2$, and let $v$ satisfy
\begin{equation} \label{Ob1}
\pa_t v + m \pa_{xxx} v = 0, \quad v(T) = v_T.
\end{equation}
Then 
\begin{equation} \label{Ob2}
\int_0^T \int_{\om} |v(t,x)|^2 \, dx \, dt 
\geq C_2 \| v_T \|_{L^2_x}^2
\end{equation}
with $C_2 := C_1(T) |\om|$,
where $C_1(T)$ is the constant of Proposition \ref{lemma:Ingham}, 
and $|\om|$ is the Lebesgue measure of $\om$.
\end{lemma}

\begin{proof} Let $v_T(x) = \sum_{n \in \Z} a_n e^{inx}$, so that 
$v(t,x) = \sum_{n \in \Z} w_n(x) e^{im n^3 t}$ where $w_n(x) := a_n e^{i(nx - mn^3 T)}$. 
By Lemma \ref{lemma:Ingham}, for each $x \in \T$ we have
\[
\int_0^T \bigg| \sum_{n \in \Z} w_n(x) e^{imn^3 t} \bigg|^2 \, dt 
\geq C_1(T) \sum_{n \in \Z} |w_n(x)|^2 
= C_1(T) \sum_{n \in \Z} |a_n|^2 
= C_1(T) \| v_T \|_{L^2(\T)}^2 \,,
\]
then we integrate over $x \in \om$. 
\end{proof}

\begin{lemma}[Observability of $\ttL_5 := \pa_t + m \pa_{xxx} + \mR$] 
\label{lemma:Obs2}
Let $T > 0$, let $\om \subset \T$ be an open set and let $m \geq 1/2$. 
Let $\mR \in C([0,T], \mL(L^2_x))$, with 
$\| \mR(t)h \|_0 \leq r_0 \| h \|_0$ for all $h \in L^2_x$, 
all $t \in [0,T]$, where $r_0$ is a positive constant.  
Let $v_T \in L^2(\T)$ and let $v \in C([0,T], L^2_x)$ be the solution of the Cauchy problem 
\begin{equation} \label{Ob5}
\pa_t v + m \pa_{xxx} v + \mR v = 0, \quad v(T) = v_T,
\end{equation}
which is globally wellposed by Lemma \ref{lemma:App2}$(iii)$. 
Then 
\[
\int_0^T \int_{\om} |v(t,x)|^2 \, dx \, dt 
\geq C_3 \| v_T \|_{L^2_x}^2
\]
with $C_3 := C_2/4$, 
provided that $r_0$ is small enough 
(more precisely, $r_0$ smaller than a constant depending 
only on $T,C_2$ where $C_2$ is the constant in Lemma \ref{lemma:Obs1}).
\end{lemma}

\begin{proof} 
Let $v_1$ be the solution of $\pa_t v_1 + m \pa_{xxx} v_1 = 0$,  $v_1(T) = v_T$, and 
let $v_2 := v - v_1$. Then $v_2$ solves 
\begin{equation} \label{Ob6}
(\pa_t + m \pa_{xxx} + \mR) v_2 = - \mR v_1, \quad v_2(T) = 0. 
\end{equation}
By \eqref{App19}, applied for $s=0$, $\a = 0$, $f = - \mR v_1$, we get 
\begin{equation}  \label{Ob6.1}
\| v_2 \|_{T,0} 
\leq 2^{4 T r_0} 4 T \| \mR v_1 \|_{T,0}
\leq 2^{4 T r_0} 4 T r_0 \| v_T \|_0.
\end{equation}
Using the elementary inequality $(a+b)^2 \geq \frac12 a^2 - b^2$ for all $a,b \in \R$,
\[
\int_0^T \int_\om |v|^2 dx dt 
\geq \frac12 \int_0^T \int_\om |v_1|^2 dx dt 
- \int_0^T \int_\om |v_2|^2 dx dt.
\]
The integral of $|v_1|^2$ is estimated from below by \eqref{Ob2}. 
The integral of $|v_2|^2$ is bounded by $T \| v_2 \|_{T,0}^2$, then use \eqref{Ob6.1}. 
\end{proof}

\begin{lemma}[Observability of $\ttL_4 := \pa_t + m \pa_{xxx} + a_{14}(t,x) \pa_x + a_{15}(t,x)$, 
$a_{14}$ with zero mean]
\label{lemma:Obs3}
There exists a universal constant $\s>0$ with the following property.
Let $T > 0$, and let $\om \subset \T$ be an open set.
Let $m \geq 1/2$ and
let $a_{14}(t,x)$, $a_{15}(t,x)$ be two functions, with 
$a_{14}, \pa_t a_{14}, a_{15} \in C([0,T],H^\s_x)$, 
\begin{equation}\label{ab picc}
\int_\T a_{14}(t,x) \, dx = 0 \ \  \forall t \in [0,T], \qquad   
\| a_{14}, \pa_t a_{14}, a_{15} \|_{T,\s} \leq \d.
\end{equation}
Let $v_T \in L^2(\T)$ and let $v \in C([0,T], L^2_x)$ be the solution of the Cauchy problem 
\begin{equation} \label{nob.8}
\ttL_4 v = 0, \quad v(T) = v_T,
\end{equation}
which is globally wellposed by Lemma \ref{lemma:App3}. Then 
\[
\int_0^T \int_{\om} |v(t,x)|^2 \, dx \, dt 
\geq C_4 \| v_T \|_{L^2_x}^2
\]
with $C_4 := C_3/16$, 
provided that $\d$ is small enough 
(more precisely, $\d$ smaller than a constant depending 
only on $T,C_3$). 
\end{lemma}

\begin{proof}
Following the procedure of Section \ref{step-5}, 
we consider the transformation $\mS$ in \eqref{L30}, \eqref{L33}, 
which conjugates $\ttL_4$ to 
\[
\ttL_5 := \mS^{-1} \ttL_4 \mS 
= \pa_t + m \pa_{xxx} + \mR, 
\]
where the operator $\mR$ is defined in \eqref{L34}, \eqref{L32}, 
it belongs to $C([0,T], \mL(L^2_x))$, and satisfies the bounds in Lemma \ref{lemma:mS}.
Let $v$ be the solution of \eqref{nob.8}, and define $\tilde v := \mS^{-1} v$. 
Then $\tilde v$ solves 
$\ttL_5 \tilde v = 0$, $\tilde v(T) = \tilde v_T$ 
where $\tilde v_T := \mS^{-1}(T) v_T$, 
and therefore Lemma \ref{lemma:Obs2} applies to $\tilde v$ if $\d$ is sufficiently small. 
By Lemmas \ref{lemma:mS}, \ref{lemma:App3} and Remark \ref{rem:Lk} we get 
\[
\int_0^T \int_\om |(\mS^{-1} - I)v|^2 \, dx \, dt 
\leq T \| (\mS^{-1} - I)v \|_{T,0}^2  
\leq C \d^2 \| v \|_{T,0}^2 
\leq C' \d^2 \| v_T \|_0^2 
\]
for some constant $C'$ depending on $T$. 
We split $\tilde v = v + (\mS^{-1} - I)v$, and we get
\[
\int_0^T \int_\om | \tilde v|^2 \, dx \, dt 
\leq 2 \int_0^T \int_\om |v|^2 \, dx \, dt + 2 C' \d^2 \| v_T \|_0^2.
\]
Moreover $\| v_T \|_0 = \| \mS(T) v_T \|_0 \leq 2 \| \tilde v_T \|_0$, 
and the thesis follows for $\d$ small enough.
\end{proof}

\begin{lemma}[Observability of $\ttL_3 := \pa_t + m \pa_{xxx} + a_{12}(t,x) \pa_x + a_{13}(t,x)$]
\label{lemma:Obs4}
There exists a universal constant $\s>0$ with the following property.
Let $T > 0$, and let $\om \subset \T$ be an open set and let $m \geq 1/2$. 
Let $a_{12}(t,x)$, $a_{13}(t,x)$ be two functions, with 
$a_{12}, \pa_t a_{12}, a_{13} \in C([0,T],H^\s_x)$, 
\begin{equation}\label{cd piccoli}
\| a_{12}, \pa_t a_{12}, a_{13} \|_{T,\s} \leq \d.
\end{equation}
Let $v_T \in L^2(\T)$ and let $v \in C([0,T], L^2_x)$ be the solution of the Cauchy problem 
\begin{equation} \label{nob.10}
\ttL_3 v = 0, \quad v(T) = v_T,
\end{equation}
which is globally wellposed by Lemma \ref{lemma:App4}. Then 
\begin{equation} \label{nob.11}
\int_0^T \int_{\om} |v(t,x)|^2 \, dx \, dt 
\geq C_5 \| v_T \|_{L^2_x}^2
\end{equation}
for some $C_5 > 0$ depending on $T,\om$, provided that $\d$ in \eqref{cd piccoli} is sufficiently small (more precisely, $\d$ smaller than a constant depending on $T,\om,C_4$).
\end{lemma}

\begin{proof}
Following the procedure of Section \ref{step-4}, 
we consider the transformation $\mT$ defined in \eqref{L25}, \eqref{L28}, 
which conjugates $\ttL_3$ to 
\[
\ttL_4 := \mT^{-1} \ttL_3 \mT 
= \pa_t + m \pa_{xxx} + a_{14}(t,x) \pa_x + a_{15}(t,x),
\]
where $a_{14}, a_{15}$ are defined in \eqref{L27}, 
and $\int_\T a_{14}(t,x) \, dx = 0$. 
By \eqref{m.13}, the function $p$ defined in \eqref{L28} satisfies $|p(t)| \leq C \d$ for all $t \in [0,T]$.  
Let $v$ be the solution of the Cauchy problem \eqref{nob.10}. 
Then $\tilde v := \mT^{-1} v$ solves $\ttL_4 \tilde v = 0$, $\tilde v(T) = \mT^{-1} (T) v_T$. 
Let $\om_1 = [\a_1, \b_1]$ be an interval contained in $\om$. 
For $\d$ small enough, one has 
\[
[\a_1 - p(t), \b_1 - p(t)]
\subseteq [\a_1 - \d, \b_1 + \d] \subset \om 
\quad \forall t \in [0,T].
\]
The change of variable $x - p(t) = y$, $dx = dy$ gives
\[
\int_0^T \int_{\om_1} |\tilde v(t,x)|^2 \, dx \, dt
= \int_0^T \int_{\a_1 - p(t)}^{\b_1 - p(t)} |v(t,y)|^2 \, dy \, dt
\leq \int_0^T \int_\om |v(t,y)|^2 \, dy \, dt.
\]
By \eqref{m.13}, for $\d$ small enough, 
Lemma \ref{lemma:Obs3} can be applied to $\tilde v$ on the interval $\om_1$ and the thesis follows, since $\| \tilde v (T) \|_0 = \| \mT\inv(T) v_T \|_0 = \| v_T \|_0$. 
\end{proof}

\begin{lemma}[Observability of $\ttL_2 := \pa_t + m \pa_{xxx} + a_8 (t,x) \pa_{xx} + a_9 (t,x) \pa_x + a_{10} (t,x)$]
\label{lemma:Obs4.1} 
There exists a universal constant $\s>0$ with the following property.
Let $T > 0$, and let $\om \subset \T$ be an open set and let $m \geq 1/2$. 
Let $a_8(t,x)$, $a_9(t,x)$, $a_{10}(t,x)$ be three functions, with 
$a_8, \pa_t a_8, a_9, \pa_t a_9, a_{10} \in C([0,T],H^\s_x)$, 
\begin{equation}\label{a89dieci}
\int_\T a_8(t,x) \, dx = 0 \ \  \forall t \in [0,T], \qquad   
\| a_8, \pa_t a_8, a_9, \pa_t a_9, a_{10} \|_{T,\s} \leq \d.
\end{equation}
Let $v_T \in L^2(\T)$ and let $v \in C([0,T], L^2_x)$ be the solution of the Cauchy problem 
\begin{equation} \label{cau.1}
\ttL_2 v = 0, \quad v(T) = v_T,
\end{equation}
which is globally wellposed by Lemma \ref{lemma:App5}. Then 
\begin{equation} \label{bob.1}
\int_0^T \int_{\om} |v(t,x)|^2 \, dx \, dt 
\geq C_6 \| v_T \|_{L^2_x}^2
\end{equation}
for some $C_6 > 0$ depending on $T,\om$, provided that $\d$ in \eqref{a89dieci} is sufficiently small 
(more precisely, $\d$ smaller than a constant depending on $T,\om,C_5$).
\end{lemma}

\begin{proof}
Following the procedure of Section \ref{step-3}, 
we consider the multiplication operator $\mM$ 
defined in \eqref{M1}, \eqref{M5}, 
which conjugates $\ttL_2$ to 
\[
\mM^{-1} \ttL_2 \mM = \ttL_3, \quad 
\ttL_3 = \pa_t + m \pa_{xxx} + a_{12}(t,x) \pa_x + a_{13}(t,x),
\]
where $a_{12}, a_{13}$ are defined in \eqref{M3}.
Let $v$ be the solution of the Cauchy problem \eqref{cau.1}. 
Then $\tilde v := \mM^{-1} v$ solves $\ttL_3 \tilde v = 0$, $\tilde v(T) = \mM\inv (T) v_T$. 
Using \eqref{M10}, we have
\begin{equation*}
\int_0^T \int_\om |v(t,x)|^2 \, dx \, dt 
= \int_0^T \int_\om |\tilde v|^2 \, dx \, dt 
+ \int_0^T \int_\om |\tilde v|^2 (|q|^2 -1) \, dx \, dt
\geq (C_5 - C \d) \| v_T \|_0^2.
\end{equation*}
The first of the two integrals has been estimated from below by applying Lemma \ref{lemma:Obs4} to $\ttL_3$ (by Lemma \ref{lemma:mM}, this can be done provided that $\delta$ is sufficiently small). 
The second integral has been estimated using the bound \eqref{M10}, since $|q(t)-1| \leq C \| q-1 \|_{T,1} \leq C' \d$. 
Moreover, we have used the inequality $\| \tilde v \|_{T,0} \leq C \| \tilde v_T \|_0$ from Lemma \ref{lemma:App4}. 
The thesis follows with $C_6:= C_5/2$ by choosing $\delta$ small enough. 
\end{proof}

\begin{lemma}[Observability of $\ttL_1 := \pa_t + a_4(t) \pa_{xxx} + a_5 (t,x) \pa_{xx} + a_6(t,x) \pa_x + a_7(t,x)$] \label{lemma:Obs5}
$\phantom{.} $
There exists a universal constant $\s>0$ with the following property.
Let $T > 0$, and let $\om \subset \T$ be an open set. 
Let $a_4, a_5, a_6, a_7$ be four functions, with
$a_4\in C^1([0,T], \mathbb R)$, 
$a_5, \pa_t a_5, a_6, \pa_t a_6, a_7 \in C([0,T],H^\s_x)$, satisfying
\begin{equation}\label{a467 piccoli}
\int_\T a_5(t,x) \, dx = 0 \ \  \forall t \in [0,T], \qquad 
\| a_5, \pa_t a_5, a_6, \pa_t a_6, a_7\|_{T,\s} 
+ |a_4-1, a'_4|_T \leq \d.
\end{equation}
Let $v_T\in L^2 (\T)$ and let $v \in C([0,T], L^2_x)$ be the solution of the Cauchy problem 
\begin{equation} \label{nob.20}
\ttL_1 v = 0, \quad v(T) = v_T,
\end{equation}
which is globally wellposed by Lemma \ref{lemma:App6}. Then 
\begin{equation} \label{nob.21}
\int_0^T \int_{\om} |v(t,x)|^2 \, dx \, dt 
\geq C_7 \| v_T \|_{L^2_x}^2
\end{equation}
for some $C_7 > 0$ depending on $T,\om$, provided that $\d$ in \eqref{a467 piccoli} is sufficiently small (more precisely, $\d$ smaller than a constant depending on $T,\om,C_6$).
\end{lemma}

\begin{proof}
Following the procedure of Section \ref{step-2}, 
we consider the re-parametrization of time $\mB$ 
defined in \eqref{L17}, \eqref{L22}, 
which conjugates $\ttL_1$ to 
\[
\mB^{-1} \ttL_1 \mB = \rho \ttL_2, \quad 
\ttL_2 = \pa_\tau + m \pa_{xxx} + a_8(\tau,x) \pa_{xx} + a_9(\tau,x) \pa_x + a_{10}(\tau,x),
\]
where $\rho, a_8, a_9,a_10$ are defined in \eqref{L20}, \eqref{L24} and $\int_\T a_8(\t, x)=0$ for all $\t\in[0,T]$. 
Let $v$ be the solution of the Cauchy problem \eqref{nob.20}. 
Then $\tilde v := \mB^{-1} v$ solves $\ttL_2 \tilde v = 0$, $\tilde v(T) = \mB^{-1}(T) v_T$. 
Using \eqref{m.11}, we have
\begin{align*}
\int_0^T \int_\om |v(t,x)|^2 \, dx \, dt 
& = \int_0^T \int_\om |\tilde v(\psi(t),x)|^2 \, dx \, dt 
\\
& = \int_0^T \int_\om |\tilde v(\psi(t),x)|^2 [\psi'(t) + (1-\psi'(t))] \, dx \, dt 
\\
& = \int_0^T \int_\om |\tilde v(\tau,x)|^2 \, dx \, d\tau 
+ \int_0^T \int_\om |\tilde v(\psi(t),x)|^2 (1-\psi'(t)) \, dx \, dt 
\\
& \geq (C_6 - C \d) \| v_T \|_0^2.
\end{align*}
The first of the two integrals has been estimated from below by applying Lemma \ref{lemma:Obs4.1} to $\ttL_2$ (by Lemma \ref{lemma:mB}, this can be done provided that $\delta$ is sufficiently small). 
The second integral has been estimated using the bound \eqref{m.11} for $|\psi'(t)-1|$ 
and also the inequality $\| \tilde v \|_{T,0} \leq C \| \tilde v_T \|_0$ from Lemma \ref{lemma:App5}. 
The thesis follows with $C_7:= C_6/2$ by choosing $\delta$ small enough, since $\| \tilde v_T \|_0 = \| \mB\inv(T) v_T \|_0 = \| v_T \|_0$.
\end{proof}

\begin{lemma}
[Observability of $\ttL_0 := \pa_t + (1 + a_3) \pa_{xxx} + a_2 \pa_{xx} + a_1 \pa_x + a_0$]
\label{lemma:Obs6}
There exists a universal constant $\s>0$ with the following property.
Let $T > 0$, and let $\om \subset \T$ be an open set.
Let $c\in\R$ and $a_3(t,x), a_2(t,x), a_1(t,x), a_0(t,x)$ be four functions with 
$a_2 = c \pa_x a_3$, 
\begin{equation}\label{a3210 piccoli}
\| \pa_{tt}  a_3, \pa_t a_3, a_3, \pa_t a_1, a_1, a_0 \|_{T,\s} \leq \d.
\end{equation}
Let $v_T\in L^2 (\T)$ and let $v \in C([0,T], L^2_x)$ be the solution of the Cauchy problem 
\begin{equation} \label{nob.30}
\ttL_0 v = 0, \quad v(T) = v_T,
\end{equation}
which is globally wellposed by Lemma \ref{lemma:App7}. Then 
\begin{equation} \label{nob.31}
\int_0^T \int_{\om} |v(t,x)|^2 \, dx \, dt 
\geq C_8 \| v_T \|_{L^2_x}^2
\end{equation}
for some $C_8 > 0$ depending on $T,\om$, provided that $\d$ in \eqref{a3210 piccoli} is sufficiently small (more precisely, $\d$ smaller than a constant depending on $T,\om,C_7$).
\end{lemma}

\begin{proof} 
Following the procedure of Section \ref{step-1}, 
we consider the transformation $\mA$ 
defined in \eqref{L4}, \eqref{L12}, \eqref{L13}, \eqref{L14}, 
which conjugates $\ttL_0$ to 
\[
\mA^{-1} \ttL_0 \mA = \ttL_1 
= \pa_t + a_4(t) \pa_{xxx} + a_5(t,x) \pa_{xx} + a_6(t,x) \pa_x + a_7(t,x)
\]
(see \eqref{L15}), where $a_4, a_5, a_6, a_7$ are defined in \eqref{L8} and $\int_\T a_5(t, x)=0$ for all $t\in[0,T]$.
Let $v$ be the solution of the Cauchy problem \eqref{nob.30}. 
Then $\tilde v := \mA^{-1} v$ solves $\ttL_1 \tilde v = 0$, $\tilde v(T) = \tilde v_T$, 
where $\tilde v_0 := \mA^{-1}(0) v_0$.  
Let $\om_1 = [\a_1, \b_1] \subset \om$. 
By \eqref{m.3} in Lemma \ref{lemma:mA}, for $\d$ sufficiently small 
Lemma \ref{lemma:Obs5} applies to $\tilde v$ on $\om_1$, and 
\[
\int_0^T \int_{\om_1} |\tilde v|^2 \, dy \, dt 
\geq C_7 \| \tilde v_T \|_0^2.
\]
By Lemma \ref{lemma:mA}, $\| v_T \|_0 = \| \mA(T) \tilde v_T \|_0 \leq C \| \tilde v_T \|_0$. 
The change of integration variable $y = x + \b(t,x)$, $dy = (1 + \b_x(t,x)) dx$ gives
\begin{align*}
\int_0^T \int_{\om_1} |\tilde v|^2 \, dy \, dt
& = \int_0^T \int_{\om_1} | (\mA^{-1} v)(t,y)|^2 \, dy \, dt
\\
& = \int_0^T \int_{\om_2(t)} \frac{|v(t,x)|^2}{1 + \b_x(t,x)} \, dx \, dt 
\leq 2 \int_0^T \int_{\om} |v(t,x)|^2 \, dx \, dt,
\end{align*}
where $\om_2(t) := \{ x : x + \b(t,x) \in \om_1 \}$. 
We have used the fact that, for $\d$ small enough, $\om_2(t) \subset \om$, 
and the bound \eqref{m.3} for $|\b_x(t,x)| \leq C \| \b \|_{T,2} \leq C' \d$. 
\end{proof}

\section{Controllability}
\label{section:con}

In this section we prove the controllability of the linearized operator 
$\ttL_0$, using its observability (Lemma \ref{lemma:Obs6}), by means of the HUM method. 
We also prove higher regularity of the control.

\begin{lemma}
[Controllability of $\ttL_0$] 
\label{lemma:Con1}
Let $T > 0$, and let $\om \subset \T$ be an open set.
Let $a_3, a_2, a_1, a_0$ be four functions of $(t,x)$ with 
$a_2 = 2 \pa_x a_3$ satisfying \eqref{a3210 piccoli}.
Let $\ttL_0$ be the linear operator 
\begin{equation} \label{ttL0}
\ttL_0 := \pa_t + (1 + a_3) \pa_{xxx} + a_2 \pa_{xx} + a_1 \pa_x + a_0.
\end{equation}

\emph{(i) Existence.} 
There exist constants $\d_0, C$ such that, 
if $\d$ in \eqref{a3210 piccoli} is smaller than $\d_0$, 
then the following property holds.
Given any three functions $g_1(t,x)$, $g_2(x)$, $g_3(x)$, 
with $g_1 \in C([0,T],L^2_x)$, $g_2, g_3 \in L^2_x$, 
there exists a function $\ph \in C([0,T], L^2_x)$ 
such that the solution $h$ of the Cauchy problem 
\begin{equation} \label{C1}
\ttL_0 h = g_1 + \chi_\om \ph, \quad h(0) = g_2
\end{equation}
satisfies $h(T) = g_3$. 
(Note that the Cauchy problem \eqref{C1} is globally well-posed by Lemma \ref{lemma:App7}). 
Moreover 
\begin{equation} \label{C2}
\| \ph \|_{T,0} \leq C ( \| g_1 \|_{T,0} + \| g_2 \|_0 + \| g_3 \|_0 ).
\end{equation}

\emph{(ii) Uniqueness.} Let $\ttL_0^*$ be the linear operator 
\begin{equation} \label{C2.1}
\ttL_0^* \psi := - \pa_t \psi - \pa_{xxx} \{ (1 + a_3) \psi \} 
+ \pa_{xx} (a_2 \psi) - \pa_x (a_1 \psi) + a_0 \psi.
\end{equation}
The control $\ph$ in \emph{(i)} is the unique solution of the equation
$\ttL_0^* \ph = 0$ such that the solution $h$ of the Cauchy problem \eqref{C1} satisfies $h(T) = g_3$.
\end{lemma}

The proof of Lemma \ref{lemma:Con1} is given below, 
and it is based on the following classical lemma.
In this section we use the standard notation $\langle u , v \rangle := \int_\T u v \, dx$.

\begin{lemma} \label{lemma:LM} 
Let $a_3,a_2,a_1,a_0$ be functions satisfying \eqref{a3210 piccoli} and $a_2 = 2 \pa_x a_3$. 
Let $\ttL_0^*$ be the operator defined in \eqref{C2.1}. 
For every $(g_1, g_2, g_3)$ with $g_1 \in C([0,T],L^2_x)$, $g_2, g_3 \in L^2_x$ 
there exists a unique $\ph_1 \in L^2_x$ such that 
for all $\psi_1 \in L^2_x$, the solutions $\ph, \psi \in C([0,T],L^2_x)$ of the Cauchy problems
\begin{equation}  \label{C3.1}
\begin{cases}
\ttL_0^* \ph = 0 \\
\ph(T) = \ph_1
\end{cases}
\qquad 
\begin{cases}
\ttL_0^* \psi = 0 \\
\psi(T) = \psi_1
\end{cases}
\end{equation}
satisfy
\begin{equation}\label{LM1}
\int_0^T \langle g_1 + \chi_\om \ph , \psi \rangle \, dt 
+ \langle g_2 , \psi(0) \rangle 
- \langle g_3 , \psi(T) \rangle = 0 
\end{equation}
(note that the global well-posedness of the Cauchy problems \eqref{C3.1} follows 
from Lemma \ref{lemma:App7} and Remark \ref{rem:Lk}).
Moreover $\ph$ satisfies \eqref{C2}. 
\end{lemma}

\begin{proof} 
Given $\ph_1, \psi_1 \in L^2_x$, let $\ph, \psi$ be the solutions of the Cauchy problems
\eqref{C3.1}, and define 
\begin{equation} \label{LM2.1} 
B(\ph_1 , \psi_1) := \int_0^T \langle \chi_\om \ph, \psi \rangle\,dt,
\qquad 
\Lambda(\psi_1) := \langle g_3 , \psi(T) \rangle
- \langle g_2 , \psi(0) \rangle
- \int_0^T \langle g_1, \psi \rangle\,dt.
\end{equation}
The bilinear map $B : L^2_x\times L^2_x \to \R$ is well defined 
and continuous because $|\chi_\om(x)| \leq 1$ and, 
by Lemma \ref{lemma:App7} and Remark \ref{rem:Lk}, 
$\| \ph \|_{T,0} \leq C \| \ph_1 \|_0$, and similarly for $\psi$.  
Moreover $B$ is coercive by Lemma \ref{lemma:Obs6} and Remark \ref{come fosse Ham}. 
The linear functional $\Lambda$ is bounded, with 
\[
| \Lambda(\psi_1) | \leq C \| g \|_{T,0} \| \psi_1 \|_0
\quad \forall \psi_1 \in L^2_x, \qquad 
\| g \|_{T,0} := \| g_1 \|_{T,0} + \| g_2 \|_0 + \| g_3 \|_0.
\]
Thus, by Riesz representation theorem (or Lax-Milgram),
there exists a unique $\ph_1\in L^2_x$ such that 
\begin{equation} \label{LM3}
B(\ph_1, \psi_1) = \Lambda(\psi_1) \quad \forall \psi_1 \in L^2_x.
\end{equation}
Moreover $\| \ph_1 \|_0 \leq C \|\Lambda\|_{\mL(L^2_x,\R)} \leq C' \| g \|_{T,0}$. 
Since $\| \ph \|_{T,0} \leq C \| \ph_1 \|_0$, we get \eqref{C2}. 
\end{proof}

\begin{proof}[Proof of Lemma \ref{lemma:Con1}]
$(i)$. 
Let $\ph_1 \in L^2_x$ be the unique solution of \eqref{LM3} given by Lemma \ref{lemma:LM}. 
Consider any $\psi_1 \in L^2_x$, 
and let $\ph, \psi \in C([0,T],L^2_x)$ 
be the unique solutions of the Cauchy problems \eqref{C3.1}. 
Recalling \eqref{LM1}, \eqref{C1} and integrating by parts, we have
\begin{align*}
0 & =  
\int_0^T \langle g_1 + \chi_\om \ph, \psi \rangle \, dt
+ \langle g_2, \psi(0) \rangle
- \langle g_3, \psi(T) \rangle 
\\
& = \int_0^T \langle \ttL_0 h, \psi \rangle \, dt
+ \langle g_2, \psi(0) \rangle
- \langle g_3, \psi(T) \rangle
\\
& = \langle h(T), \psi(T) \rangle
- \langle h(0), \psi(0) \rangle
+ \int_0^T \langle h, \ttL_0^* \psi \rangle \, dt
+ \langle g_2, \psi(0) \rangle
- \langle g_3, \psi(T) \rangle
\\
& = \langle h(T), \psi(T) \rangle
- \langle g_3, \psi(T) \rangle
\\
& = \langle h(T) - g_3, \psi_1 \rangle,
\end{align*}
from which it follows that $h(T) = g_3$.

$(ii)$. Assume that $\tilde \ph \in C([0,T],L^2_x)$ satisfies $\ttL_0^* \tilde \ph = 0$ 
and it has the property that the solution $h$ of the Cauchy problem \eqref{C1} 
satisfies $h(T) = g_3$. 
Let $\tilde \ph_1 := \tilde \ph(T)$. 
The same integration by parts as above shows that 
$B(\tilde \ph_1, \psi_1) = \Lambda(\psi_1)$ for all $\psi_1 \in L^2_x$. 
By the uniqueness in Lemma \ref{lemma:LM}, $\tilde \ph_1 = \ph_1$. 
\end{proof}

\begin{lemma}[Higher regularity] \label{lemma:Con2}
Let $T,\om,a_3,a_2,a_1,a_0,\mL_0,g_1,g_2,g_3$ be as in Lemma \ref{lemma:Con1}. 
There exist two positive constants $\d_*,\s$ with the following property.
Let $s>0$ be given. 
Assume that $a_0, a_1, a_2, a_3 \in C^2 ([0,T], H^{s+\s}_x)$. 
Let
\[
\d(\mu) := \sum_{k=0,1,2, \ i = 0,1,2,3} \| \pa_t^k a_i \|_{T,\mu+\s}, \quad \mu \in [0,s].
\]
Let $\| g \|_{T,s}:=\| g_1 \|_{T,s} + \| g_2 \|_s + \| g_3 \|_s < \infty$.
If $\d(0) \leq \d_*$, then the control $\ph$ constructed in Lemma \ref{lemma:Con1} 
and the solution $h$ of \eqref{C1} satisfy
\begin{equation} \label{C3 tame}
\| \ph, h \|_{T,s} 
\leq C_s (\| g \|_{T,s} + \d(s) \| g \|_{T,0})
\end{equation}
for some positive $C_s$ depending on $s,T,\om$.
Moreover, if $g_1 \in C^1([0,T],H^s_x)$, then 
\begin{equation} \label{C4}
\| \pa_t \ph, \pa_t h \|_{T,s+3} + \| \pa_{tt} \ph, \pa_{tt} h \|_{T,s}
\leq C_s \{ \| g \|_{T,s+6} + \| \pa_t g_1 \|_{T,s} 
+ \d(s) \| g \|_{T,6} \}.
\end{equation}
\end{lemma}

\begin{proof}
Let $g_1 \in C([0,T], H^s_x)$, $g_2, g_3 \in H^s_x$. 
Let $\ph, h \in C([0,T], L^2_x)$ be the solution of the control problem 
constructed in Lemma \ref{lemma:Con1}, namely
\begin{equation} \label{lug.2}
\ttL_0^* \ph = 0, \quad 
\ttL_0 h = \chi_\om \ph + g_1, \quad
h(0) = g_2, \quad 
h(T) = g_3.
\end{equation}
To prove that $h, \ph \in C([0,T], H^s_x)$, it is convenient to use the transformations 
of Section \ref{sec:regu}, to prove higher regularity for the solution $\tilde h, \tilde \ph$
of the transformed control problem, 
and then to go back to $h,\ph$ proving their higher regularity. 
Recall that 
\begin{equation} \label{lug.2.bis}
\ttL_0 = \mA \mB \rho \mM \mT \mS \ttL_5 \mS^{-1} \mT^{-1} \mM^{-1} \mB^{-1} \mA^{-1},
\end{equation} 
where $\ttL_5 = \pa_t + m \pa_{xxx} + \mR$ 
and $\mA, \mB, \rho, \mM, \mT, \mS$ are defined in Section \ref{sec:regu}.
In particular, 
\begin{itemize}
\item[$\cdot$]
$\mA$ is the change of the space variable 
$(\mA h)(t,x) = h(t, x + \b(t,x))$ (see \eqref{L4}), 
where $\b$ is defined in \eqref{L14}, \eqref{L12}, \eqref{L13};

\item[$\cdot$]
$\mB$ is the reparametrization of time 
$(\mB h)(t,x) = h(\psi(t),x)$ (see \eqref{L17}), 
where $\psi$ is defined in \eqref{L22};

\item[$\cdot$]
$\rho(t)$ is the function defined in \eqref{L18};

\item[$\cdot$]
$\mM$ is the multiplication operator $(\mM h)(t,x) = q(t,x) h(t,x)$ (see \eqref{M1}),
where $q$ is defined in \eqref{M5};

\item[$\cdot$]
$\mT$ is the translation of the space variable 
$(\mT h)(t,x) = h(t, x + p(t))$ (see \eqref{L25}), 
where $p$ is defined in \eqref{L28};

\item[$\cdot$]
$\mS$ is the pseudo-differential operator 
$(\mS h)(t,x) = h(t,x) + \g(t,x) \pa_x^{-1} h(t,x)$ (see \eqref{L30}), 
where $\g$ is defined in \eqref{L33} 
and $\pa_x^{-1} h$ is the primitive of $h$ with zero average in $x$ (defined in Fourier);

\item[$\cdot$]
$\mR$ is the bounded operator defined in \eqref{L34}.
\end{itemize}
Let 
\begin{equation} \label{lug.3}
\ttL_5^* := - \pa_t - m \pa_{xxx} + \mR^T,
\end{equation}
where $\mR^T$ is the $L^2_x$-adjoint of $\mR$. 
Let
\begin{alignat}{2}
\tilde h & := (\mA \mB \mM \mT \mS)^{-1} h, \quad &
\tilde g_1 & := (\mA \mB \rho \mM \mT \mS)^{-1} g_1, 
\notag \\ 
\tilde g_2 & := (\mA \mB \mM \mT \mS)^{-1}|_{t=0} \, g_2, \quad &
\tilde g_3 & := (\mA \mB \mM \mT \mS)^{-1}|_{t=T} \, g_3, \quad 
\label{lug.5} \\
\tilde \ph & := \mS^T \mT^T \mM^T \mB^{-1} \mA^T \ph, \qquad & 
K \tilde \ph & := (\mA \mB \rho \mM \mT \mS)^{-1} (\chi_\om (\mS^T \mT^T \mM^T \mB^{-1} \mA^T)^{-1} \tilde \ph). 
\notag 
\end{alignat}
Note that, except for $\mS^{-1}, \mS^{-T}$, the operator $K$ is a multiplication operator, namely 
\begin{equation} \label{lug.6}
K \tilde \ph = \mS^{-1} ( \zeta \mS^{-T} \tilde \ph), 
\quad \text{where} \quad 
\zeta(t,x) := \rho\inv \mT\inv \mM^{-2} \mB\inv \mA\inv [(1+ \b_x) \chi_\om].
\end{equation}
Since $h,\ph \in C([0,T], L^2_x)$, and $g_1 \in C([0,T],H^s_x)$, $g_2, g_3 \in H^s_x$, 
by \eqref{lug.5} and the estimates for $\mA, \mB, \rho, \mM, \mT, \mS$ in Section \ref{sec:regu},
one has
\[
\tilde h, \tilde \ph, K \tilde \ph \in C([0,T], L^2_x), \quad 
\tilde g_1 \in C([0,T],H^s_x), \quad 
\tilde g_2, \tilde g_3 \in H^s_x.
\]
Since $h,\ph$ satisfy \eqref{lug.2}, one proves that $\tilde h, \tilde \ph$ satisfy 
\begin{equation} \label{lug.10}
\ttL_5^* \tilde \ph = 0, \quad 
\ttL_5 \tilde h = K \tilde \ph + \tilde g_1, \quad
\tilde h(0) = \tilde g_2, \quad 
\tilde h(T) = \tilde g_3.
\end{equation}
The last three equations in \eqref{lug.10} are straightforward. 
To prove that $\ttL_5^* \tilde \ph = 0$,  
we start from the equality 
\[
\la \ph(T), v(T) \ra - \la \ph(0) , v(0) \ra 
= \int_0^T \la \ph, \ttL_0 v \ra dt \quad \forall v \in C^\infty([0,T] \times \T)
\]
(which is a weak form of $\ttL_0^* \ph = 0$), we recall \eqref{lug.2.bis}, 
and apply all the changes of variables $\mA, \mB, \mM, \mT, \mS$ in the integral.
Thus $\tilde h, \tilde \ph$ solve this control problem: 
\begin{equation} \label{lug.11}
\begin{cases}
\text{Given $\tilde g_1, \tilde g_2, \tilde g_3$, 
find $\tilde \ph$ such that the solution $\tilde h$} \\
\text{of the Cauchy problem $\ttL_5 \tilde h = K \tilde \ph + \tilde g_1$, $\tilde h(0) = \tilde g_2$} \\
\text{satisfies $\tilde h(T) = \tilde g_3$, and moreover $\tilde \ph$ solves $\ttL_5^* \tilde \ph = 0$.}
\end{cases}
\end{equation}
The function $\tilde \ph$ is the unique solution of \eqref{lug.11}. 
To prove it, assume that $\tilde \ph_{bis} \in C([0,T], L^2_x)$ solves \eqref{lug.11}, 
and let $\tilde h_{bis}$ be the solution of the corresponding Cauchy problem 
$\ttL_5 \tilde h_{bis} = K \tilde \ph_{bis} + \tilde g_1$, $\tilde h_{bis}(0) = \tilde g_2$.
Define 
\[
h_{bis} := \mA \mB \mM \mT \mS \tilde h_{bis}, \quad 
\ph_{bis} := \mA^{-T} \mB \mM^{-T} \mT^{-T} \mS^{-T} \tilde \ph_{bis}.
\]
Then $h_{bis}, \ph_{bis}$ solve \eqref{lug.2}. 
By the uniqueness in Lemma \ref{lemma:Con1}(ii) it follows that $\ph_{bis} = \ph$, $h_{bis} = h$. 
Therefore $\tilde \ph_{bis} = \tilde \ph$ and $\tilde h_{bis} = \tilde h$. 

Now we prove that $\tilde h, \tilde \ph \in C([0,T], H^s_x)$. 
We follow an argument used by Dehman-Lebeau \cite[Lemma 4.2]{Dehman-Lebeau}, Laurent \cite[Lemma 3.1]{Laurent}, and \cite[Proposition 8.1]{ABH}.
First, we prove the thesis for $\tilde g_1=0$, $\tilde g_3=0$. 
Consider the map 
\begin{equation}\label{hr2}
S \colon L^2_x \to L^2_x, \quad 
S \tilde \ph_1 = \tilde h(0)
\end{equation} 
obtained by the composition 
$\tilde \ph_1 \mapsto \tilde \ph \mapsto \tilde h \mapsto \tilde h(0)$, 
where $\tilde \ph, \tilde h$ are the solutions of the Cauchy problems
\begin{equation}  \label{hr3}
\begin{cases}
\ttL_5^* \tilde \ph = 0 \\ 
\tilde \ph(T) = \tilde \ph_1,
\end{cases}
\qquad 
\begin{cases}
\ttL_5 \tilde h = K \tilde \ph \\ 
\tilde h(T) = 0.
\end{cases}
\end{equation}
From the existence and uniqueness of $\tilde \ph_1 \in L^2_x$ such that $\tilde \ph$ solves \eqref{lug.11}
it follows that $S$ is an isomorphism of $L^2_x$. 
The initial datum $\tilde g_2$ is given, 
so we fix $\tilde \ph_1 \in L^2_x$ such that $S \tilde \ph_1 = \tilde g_2$.
We have to estimate $\| \Lm^s \tilde \ph_1 \|_0 
\leq C \| S \Lm^s \tilde \ph_1 \|_0$, 
where $\Lm^s$ is the Fourier multiplier of symbol $\la \xi \ra^{s} := (1 + \xi^2)^{s/2}$, 
$s > 0$. 
To study the commutator $[S,\Lm^s]$, 
we compare $(\Lm^s \tilde \ph, \Lm^s \tilde h)$ with $(\bar \ph, \bar h)$ defined by 
\begin{equation}  \label{hr5}
\begin{cases}
\ttL_5^* \bar \ph = 0 \\ 
\bar \ph(T) = \Lm^s \ph_1,
\end{cases}
\qquad 
\begin{cases}
\ttL_5 \bar h = K \bar \ph \\ 
\bar h(T) = 0.
\end{cases}
\end{equation}
The difference $\Lm^s \tilde \ph - \bar \ph$ satisfies
\begin{equation}  \label{hr6}
\begin{cases}
\ttL_5^* (\Lm^s \tilde \ph - \bar \ph) = \mF_1, 
\\
(\Lm^s \tilde \ph - \bar \ph)(T) = 0
\end{cases}
\quad \text{where} \quad 
\mF_1 := [\ttL_5^*, \Lm^s] \tilde \ph
= [\mR^T, \Lm^s] \tilde \ph.
\end{equation}
From Lemma \ref{lemma:App2} and Remark \ref{rem:Lk}, 
$\| \Lm^s \tilde \ph - \bar \ph \|_{T,0} \leq C \| \mF_1 \|_{T,0}$.
We recall the classical estimate for the commutator of $\Lm^s$ and any multiplication operator 
$h \mapsto ah$:
\begin{equation} \label{lug.50}
\| [\Lm^s, a] h \|_0 
\leq C_s (\| a \|_{2} \| h \|_{s-1} + \| a \|_{s+1} \| h \|_{0}).
\end{equation}
By \eqref{lug.50} and formulas \eqref{L30}, \eqref{L33}, \eqref{L34}, 
the commutator $\mF_1 = [\mR^T, \Lm^s] \tilde \ph$ satisfies 
\begin{align} 
\| \mF_1 \|_{T,0} 
& \leq C_s ( \| a_{14}, a_{17}, a_{18} \|_{T,\s} \| \tilde \ph \|_{T,s-1}
+ \| a_{14}, a_{17}, a_{18} \|_{T,s+\s} \| \tilde \ph \|_{T,0}) 
\notag \\ & 
\leq C_s ( \d(0) \| \tilde \ph \|_{T,s-1} 
+ \d(s) \| \tilde \ph \|_{T,0}).
\label{lug.12}
\end{align}
The difference $\Lm^s \tilde h - \bar h$ satisfies
\begin{equation}  \label{hr7}
\begin{cases}
\ttL_5 (\Lm^s \tilde h - \bar h) = K (\Lm^s \tilde \ph - \bar \ph) + \mF_2, 
\\
(\Lm^s \tilde h - \bar h)(T) = 0,
\end{cases}
\quad \text{where} \quad
\mF_2 := [\mR^T, \Lm^s] \tilde h + [\Lm^s, K] \tilde \ph.
\end{equation}
We have $\| K (\Lm^s \tilde \ph - \bar \ph) \|_{T,0} 
\leq C \| \Lm^s \tilde \ph - \bar \ph \|_{T,0} 
\leq C \| \mF_1 \|_{T,0}$, and therefore, by Lemma \ref{lemma:App2}, 
\begin{equation}\label{hr9}
\| \Lm^s \tilde h - \bar h \|_{T,0} 
\leq C (\| \mF_1 \|_{T,0} + \| \mF_2 \|_{T,0}).
\end{equation}
Using \eqref{lug.50} and \eqref{lug.6}, we get
\begin{equation} \label{lug.13}
\| \mF_2 \|_{T,0} 
\leq C_s ( \| \tilde h, \tilde \ph \|_{T,s-1} + \d(s) \| \tilde h, \tilde \ph \|_{T,0} ).
\end{equation}
By \eqref{lug.12}, \eqref{hr9} and \eqref{lug.13} we deduce that 
\[
\| \Lm^s \tilde h - \bar h \|_{T,0} 
\leq C_s ( \| \tilde h, \tilde \ph \|_{T,s-1} + \d(s) \| \tilde h, \tilde \ph \|_{T,0} ).
\]
By \eqref{hr3}, Lemma \ref{lemma:App2} and Remark \ref{rem:Lk},
\begin{equation} \label{hr9.1}
\| \tilde h, \tilde \ph \|_{T,\mu} 
\leq C_\mu \big( \| \tilde \ph \|_{T,\mu} + \d(\mu) \| \tilde \ph \|_{T,0} \big) 
\leq C_\mu \big( \| \tilde \ph_1 \|_{\mu} + \d(\mu) \| \tilde \ph_1 \|_0 \big) , \quad \mu \geq 0. 
\end{equation}
Therefore
\begin{equation} \label{hr10}
\| (\Lm^s \tilde h - \bar h) (0) \|_0 
\leq \| \Lm^s \tilde h - \bar h \|_{T,0} 
\leq C_s (\| \tilde \ph_1 \|_{s-1} + \d(s) \| \tilde \ph_1 \|_0).
\end{equation}
Since $S \tilde \ph_1 = \tilde h(0) = \tilde g_2$, 
we have $\Lm^s \tilde h(0) = \Lm^s g_2$. 
Moreover, by the definition of $S$ in \eqref{hr2}-\eqref{hr3}, 
$\bar h(0) = S \Lm^s \tilde \ph_1$. 
Thus 
\begin{equation}\label{hr11}
\| S \Lm^s \tilde \ph_1 \|_0 
\leq \| (\Lm^s \tilde h - \bar h)(0) \|_0 + \| \Lm^s \tilde h(0) \|_0
\leq C_s ( \| \tilde \ph_1 \|_{s-1} + \d(s) \| \tilde \ph_1 \|_0 )
+ \| \tilde g_2 \|_s.
\end{equation}
Since $S$ is an isomorphism of $L^2_x$, 
$\| \Lm^s \tilde \ph_1 \|_0 \leq C \| S \Lm^s \tilde \ph_1 \|_0$, whence 
\begin{equation}\label{lug.15}
\| \tilde \ph_1 \|_s
\leq C_s ( \| \tilde g_2 \|_s + \| \tilde \ph_1 \|_{s-1} + \d(s) \| \tilde \ph_1 \|_0 ).
\end{equation}
Since $\| \tilde \ph_1 \|_0 \leq C \| \tilde g_2 \|_0$, 
by induction we deduce that 
\begin{equation} \label{hr14}
\| \tilde \ph_1 \|_s 
\leq C_s (\| \tilde g_2 \|_s + \d(s) \| \tilde g_2 \|_0 ).
\end{equation}
By \eqref{hr9.1}, we obtain 
\begin{equation} \label{hr15}
\| \tilde h, \tilde \ph \|_{T,s}
\leq C_s (\| \tilde g_2 \|_s + \d(s) \| \tilde g_2 \|_0 ),
\end{equation}
which is the thesis in the case $\tilde g_1 = 0$, $\tilde g_3 = 0$.

Now we prove the higher regularity of $\tilde h, \tilde \ph$ removing 
the assumption $\tilde g_1 = 0$, $\tilde g_3 = 0$. 
Let $\tilde g_1 \in C([0,T], H^s_x)$, $\tilde g_2, \tilde g_3 \in H^s_x$, 
and let $\tilde h, \tilde \ph$ be the solution of \eqref{lug.11}.
Let $w$ be the solution of the problem 
\[
\ttL_5 w = \tilde g_1, \quad 
w(T) = \tilde g_3.
\]
By Lemma \ref{lemma:App2}, $w \in C([0,T], H^s_x)$, with 
\begin{equation}\label{hr17}
\| w \|_{T,s} \leq C_s \{ \| \tilde g_1 \|_{T,s} + \| \tilde g_3 \|_s 
+ \d(s) ( \| \tilde g_1 \|_{T,0} + \| \tilde g_3 \|_0) \}.
\end{equation}
Let $v := \tilde h - w$. Then 
\[
\ttL_5 v = K \tilde \ph, \quad 
v(0) = \tilde g_2 - w(0), \quad 
v(T) = 0.
\]
This means that $v, \tilde \ph$ solve \eqref{lug.11} where $(\tilde g_1, \tilde g_2, \tilde g_3)$ 
are replaced by $(0, \tilde g_2 - w(0), 0)$. 
Hence \eqref{hr15} applies to $v, \tilde \ph$, and we get
\begin{equation} \label{hr19}
\| v, \tilde \ph \|_{T,s}
\leq C_s (\| \tilde g_2 - w(0) \|_s + \d(s) \| \tilde g_2 - w(0) \|_0 ).
\end{equation}
We estimate $\| \tilde g_2 - w(0) \|_s \leq \| \tilde g_2 \|_s + \| w \|_{T,s}$, 
we use \eqref{hr17} and $\| \tilde h \|_{T,s} \leq \| v \|_{T,s} + \| w \|_{T,s}$ to conclude that 
\begin{equation} \label{hr20}
\| \tilde h , \tilde \ph \|_{T,s}
\leq C_s \{ \| \tilde g \|_{T,s} 
+ \d(s) \| \tilde g \|_{T,0} \}
\end{equation}
where we have denoted, in short, 
$\| \tilde g \|_{T,s} := \| \tilde g_1 \|_{T,s} + \| \tilde g_2 \|_s + \| \tilde g_3 \|_s$. 
This proves the higher regularity for the transformed control problem \eqref{lug.11}. 
By the definitions in \eqref{lug.5}, 
\begin{align*}
\| \ph \|_{T,s} 
& \leq C_s (\| \tilde \ph \|_{T,s} + \d(s) \| \tilde \ph \|_{T,0}),
\qquad 
\| h \|_{T,s} 
\leq C_s (\| \tilde h \|_{T,s} + \d(s) \| \tilde h \|_{T,0}),
\\
\| \tilde g \|_{T,s} 
& \leq C_s (\| g \|_{T,s} + \d(s) \| g \|_{T,0}),
\end{align*}
and the proof of \eqref{C3 tame} is complete. 

The bound \eqref{C4} is deduced in a classical way from the fact that $h, \ph$ 
solve the equations 
$\ttL_0^* \ph = 0$, 
$\ttL_0 h = \chi_\om \ph + g_1$.
\end{proof}

\begin{remark}
Another possible way to prove higher regularity for $h,\ph$ is to apply the argument of 
\cite{Dehman-Lebeau, Laurent, ABH} directly to the control problem for $\ttL_0$, 
instead of passing to the transformed problem \eqref{lug.11}, applying that argument, 
and then going back to $h,\ph$. 
Such a more direct method adapted to the present case would require the construction 
of two operators $A_s, B_s$ such that 

$(i)$ $C_1 \| v \|_s \leq \| A_s v \|_0 \leq C_2 \| v \|_s$ (equivalent norm in $H^s$), 

$(ii)$ the commutator $[\ttL_0, A_s]$ is an operator of order $s-1$, 

$(iii)$ the difference $B_s \ttL_0^* - \ttL_0^* A_s$ is also of order $s-1$. 

The construction of such $A_s, B_s$ is possible, 
but probably the proof given above is more straighforward, 
and it fully exploits the advantages of conjugating $\ttL_0$ to $\ttL_5$ (Section \ref{sec:regu}).
The main point is that the commutator $[\ttL_5, \Lm^s]$ is of order $s-1$ (because $\ttL_5$ has constant coefficients up to a \emph{bounded} remainder),
while $[\ttL_0, \Lm^s]$ is of order $s+2$ (because $\ttL_0$, 
which was obtained by linearizing a \emph{quasi-linear} PDE, 
has variable coefficients also at the highest order), 
so that a modified version $A_s$ of $\Lm^s$ is needed.
\end{remark}

In view of the application of Nash-Moser theorem in section \ref{sec:proof},
we define the spaces 
\begin{equation} \label{def Es}
E_s := X_s \times X_s, \quad 
X_s := C([0,T], H^{s+6}_x) \cap C^1([0,T], H^{s+3}_x) \cap C^2([0,T], H^s_x)
\end{equation}
and 
\begin{equation} \label{def Fs}
F_s := \{ g = (g_1, g_2, g_3) : g_1 \in C([0,T], H^{s+6}_x) \cap C^1([0,T], H^s_x), 
g_2, g_3 \in H^{s+6}_x \}
\end{equation}
equipped with the norms 
\begin{equation} \label{def norm Es}
\| u,f \|_{E_s} := \| u \|_{X_s} + \| f \|_{X_s}, \quad 
\| u \|_{X_s} := \| u \|_{T,s+6} + \| \pa_t u \|_{T,s+3} + \| \pa_{tt} u \|_{T,s}
\end{equation}
and
\begin{equation} \label{def norm Fs}
\| g \|_{F_s} := \| g_1 \|_{T,s+6} + \| \pa_t g_1 \|_{T,s} + \| g_2, g_3 \|_{s+6}.
\end{equation}
With this notation, we have proved the following linear inversion result.

\begin{theorem}[Right inverse of the linearized operator] \label{thm:inv}
Let $T>0$, and let $\om \subset \T$ be an open set.
There exist two universal constants $\t,\s \geq 3$ 
and a positive constant $\d_*$ depending on $T,\om$ with the following property.

Let $s \in [0, r-\t]$, 
where $r$ is the regularity of the nonlinearity $\mN$ (see Lemma \ref{lemma:stime ai}). 
Let $g = (g_1, g_2, g_3) \in F_s$, 
and let $(u,f) \in E_{s+\s}$, 
with $\| u \|_{X_\s} \leq \d_*$. 
Then there exists $(h,\ph) := \Psi(u,f)[g] \in E_s$ such that 
\begin{equation} \label{3009}
P'(u)[h] - \chi_\om \ph = g_1, \quad 
h(0) = g_2, \quad 
h(T) = g_3,
\end{equation}
and 
\begin{equation}  \label{2809}
\| h,\ph \|_{E_s} \leq C_s \big( \| g \|_{F_s} + \| u \|_{X_{s+\s}} \| g \|_{F_0} \big)
\end{equation}
where $C_s$ depends on $s,T,\om$.
\end{theorem}

\section{Proofs} 
\label{sec:proof}

In this section we prove Theorems \ref{thm:1} and \ref{thm:byproduct}.

\subsection{Proof of Theorem \ref{thm:1}} \label{subsec:proof thm 1}

The spaces defined in \eqref{def Es}-\eqref{def norm Fs}, with $s \geq 0$, 
form scales of Banach spaces. 
We define smoothing operators $S_\theta$ in the following way. 
We fix a $C^\infty$ function $\ph : \R \to \R$ with $0 \leq \ph \leq 1$, 
\[
\ph(\xi) = 1 \quad \forall |\xi| \leq 1 \qquad \text{and} \qquad 
\ph(\xi) = 0 \quad \forall |\xi| \geq 2.
\]
For any real number $\theta \geq 1$, let $S_\theta$ be the Fourier multiplier with symbol
$\ph(\xi / \theta)$, namely
\begin{equation}
S_\theta u (x) := \sum_{k \in \Z} \hat u_k \, \ph(k/\theta) \, e^{ikx}
\qquad \text{where} \quad 
u(x) = \sum_{k \in \Z} \hat u_k e^{ikx} \in L^2(\T).
\end{equation}
The definition of $S_\theta$ extends to functions $u(t,x) = \sum_{k \in \Z} \hat u_k(t) \, e^{ikx}$ depending on time in the obvious way. 
Since $S_\theta$ and $\pa_t$ commute, the smoothing operators $S_\theta$ are defined 
on the spaces $E_s$, $F_s$ defined in \eqref{def Es}-\eqref{def Fs} by setting 
$S_\theta(u,f) := (S_\theta u, S_\theta f)$ and similarly on $g = (g_1, g_2, g_3)$. 
One easily verifies that $S_\theta$ satisfies \eqref{S1}-\eqref{S4} on $E_s$ and $F_s$. 
We define the spaces $E_a'$ with norm $\| \ \|'_a$ and 
$F_b'$ with $\| \ \|_b'$ as constructed in section \ref{sec:NM}.

We observe that $\Phi(u,f) := (P(u) - \chi_\om f, \, u(0), \, u(T) )$ 
defined in \eqref{ge2}-\eqref{ge3} belongs to $F_s$ 
when $(u,f) \in E_{s+3}$, $s \in [0, r-6]$, 
with $\| u \|_{T,4} \leq 1$. 
Its second derivative is
\[
\Phi''(u,f)[(h_1, \ph_1), (h_2, \ph_2)]
= \begin{pmatrix} P''(u)[h_1, h_2] \\ 0 \\ 0 \end{pmatrix}.
\]
For $u$ in a fixed ball $\| u \|_{X_1} \leq \d_0$, with $\d_0$ small enough, 
we estimate
\begin{equation} \label{stima Phi''}
\| P''(u)[h,w] \|_{F_s} 
\leq C_s \big( \| h \|_{X_1} \| w \|_{X_{s+3}} 
+ \| h \|_{X_{s+3}} \| w \|_{X_1}
+ \| u \|_{X_{s+3}} \| h \|_{X_1} \| w \|_{X_1} \big)
\end{equation}
for all $s \in [0, r-6]$. 
We fix $V = \{ (u,f) \in E_3 : \| (u,f) \|_{E_3} \leq \d_0 \}$,
$\d_1 = \d_*$,  
\begin{equation} \label{param.1}
a_0 = 1, \quad
\mu = 3, \quad
a_1 = \s, \quad
\a = \b = 2 \s, \quad
a_2 \in ( 3\s , r - \t]
\end{equation}
where $\d_*, \s, \t$ are given by Theorem \ref{thm:inv}, 
and $r$ is the regularity of $\mN$ in Theorem \ref{thm:1}. 
The right inverse $\Psi$ in Theorem \ref{thm:inv} satisfies the assumptions of Theorem \ref{thm:NM}.
Thus by Theorem \ref{thm:NM} we obtain that, 
if $g = (0, u_{in}, u_{end}) \in F_\b'$ with $\| g \|_{F_\b}' \leq \d$, 
then there exists a solution $(u,f) \in E_\a'$ of the equation
$\Phi(u,f) = g$, with $\| u,f \|_{E_\a}' \leq C \| g \|_{F_\b}'$ 
(and recall that $\b = \a$). 
We fix $s_1 := \a + 6$, and \eqref{stimetta} is proved. 
In fact, we have proved slightly more than \eqref{stimetta}, 
because $\| g \|_{F_\b}' \leq C \| g \|_{F_\b}$ 
and $\| u,f \|_{E_a} \leq C_a \| u,f \|_{E_\a}'$ for all $a < \a$. 

We have found a solution $(u,f)$ of the control problem \eqref{i9}-\eqref{i10}. 
Now we prove that $u$ is the unique solution of the Cauchy problem \eqref{i9}, 
with that given $f$. 
Let $u,v$ be two solutions of \eqref{i9} in $E_{s-6}$ for all $s < s_1$. 
We calculate
\[
P(u) - P(v) 
= \int_0^1 P'(v + \lm (u-v))[u-v] \, d\lm
=: \widetilde \mL_0 [u-v]
\]
where 
\[
\widetilde \mL_0 := \pa_t + (1 + \tilde a_3(t,x)) \pa_{xxx} 
+ \tilde a_2(t,x) \pa_{xx} 
+ \tilde a_1(t,x) \pa_{x}
+ \tilde a_0(t,x),
\]
\[
\tilde a_i(t,x) := \int_0^1 a_i(v+\lm(u-v))(t,x) \, d\lm, \quad i = 0,1,2,3,
\]
and $a_i(u)$ is defined in \eqref{ai}.
Note that $\tilde a_2 = 2 \pa_x \tilde a_3$ because 
$a_2(v+\lm(u-v)) = 2 \pa_x a_3(v+\lm(u-v))$ for all $\lm \in [0,1]$. 
The difference $u-v$ satisfies 
$\widetilde \mL_0 (u-v) = 0$, $(u-v)(0) = 0$.
Hence, by Lemma \ref{lemma:App7}, $u-v=0$. 
The proof of Theorem \ref{thm:1} is complete.
\qed

\subsection{Proof of Theorem \ref{thm:byproduct}} \label{subsec:proof thm byproduct}

We define 
\begin{align} \label{def Es bis}
E_s & := C([0,T], H^{s+6}_x) \cap C^1([0,T], H^{s+3}_x) \cap C^2([0,T], H^s_x),
\\
F_s & := \{ g = (g_1, g_2) : g_1 \in C([0,T], H^{s+6}_x) \cap C^1([0,T], H^s_x), 
g_2 \in H^{s+6}_x \}
\end{align}
equipped with norms 
\begin{align} \label{def norm Es bis}
\| u \|_{E_s} & := \| u \|_{T,s+6} + \| \pa_t u \|_{T,s+3} + \| \pa_{tt} u \|_{T,s}
\\
\| g \|_{F_s} & := \| g_1 \|_{T,s+6} + \| \pa_t g_1 \|_{T,s} + \| g_2 \|_{s+6},
\end{align}
and $\Phi(u) := (P(u), u(0))$. 
Given $g := (f,u_{in}) \in F_{s_0}$, the Cauchy problem \eqref{i11} writes $\Phi(u) = g$.
We fix $V, \d_1, a_0, \mu, a_1, \a, \b, a_2$ like in \eqref{param.1},
where the constants $\s,\d_*$ are now given in Lemma \ref{lemma:App7} 
and $\t=\s + 9$ by Lemma \ref{lemma:stime ai} combined with Lemma \ref{lemma:App7} and the definition of the spaces $E_s,F_s$. 
Assumption \eqref{tame in NM} about the right inverse of the linearized operator 
is satisfied by Lemmas \ref{lemma:App7} and \ref{lemma:stime ai}. 
We fix $s_0 := \a + 6$.
Then Theorem \ref{thm:NM} applies, giving the existence part of Theorem \ref{thm:byproduct}. 
The uniqueness of the solution is proved exactly as in the proof of Theorem \ref{thm:1}.
\qed

%\appendix

\section{Appendix A. Well-posedness of linear operators} 
\label{sec:WP}

\begin{lemma} \label{lemma:App1}
Let $T > 0$, $m \in \R$, $s \in \R$, $f \in C([0,T],H^s_x)$, with 
$f(t,x) = \sum_{n \in \Z} f_n(t) e^{inx}$. 
Let $A$ be the linear operator defined by $A f := v$ where $v$ is the solution of 
\begin{equation} \label{App1}
\begin{cases}
\pa_t v + m \pa_{xxx} v = f \quad \forall (t,x) \in [0,T] \times \T, \\
v(0,x) = 0. \end{cases}
\end{equation}
Then 
\begin{equation} \label{App2}
A f(t,x) = \sum_{n \in \Z} (Af)_n(t) e^{inx}, \quad 
(Af)_n(t) = \int_0^t e^{i m n^3 (\t-t)} f_n(\t) \, d\t,
\end{equation}
$Af$ belongs to $C([0,T],H^s_x) \cap C^1([0,T],H^{s-3}_x)$, and 
\begin{equation} \label{App3}
\| A f \|_{T,s} \leq T \| f \|_{T,s} \,.
\end{equation}
\end{lemma}

\begin{proof}
Formula \eqref{App2} simply comes from variation of constants. 
By H\"older's inequality,
\[
|(Af)_n(t)| \leq \sqrt{t} \Big( \int_0^t |f_n(\t)|^2 \, d\t \Big)^{\frac12} \quad \forall t \in [0,T]
\]
and therefore, for each $t \in [0,T]$, 
\begin{align*}
\| Af(t) \|_{H^s_x}^2 
& = \sum_{n \in \Z} |(Af)_n(t)|^2 \langle n \rangle^{2s} 
\leq \sum_{n \in \Z} t \int_0^t |f_n(\t)|^2 \, d\t  \langle n \rangle^{2s} 
\\ 
&  \leq t \int_0^t \sum_{n \in \Z} |f_n(\t)|^2 \langle n \rangle^{2s} \, d\t
= t \int_0^t \| f(\t) \|_{H^s_x}^2 \, d\t
\leq t^2 \| f \|_{C([0,t],H^s_x)}^2.
\end{align*}
Taking the sup over $t \in [0,T]$ we get the thesis.
\end{proof}

We remark that for $s \leq 3$ the operator $A$ is well-defined in the sense of distributions.
We also recall that $\mL(H^s_x)$ is the space of linear bounded operators of $H^s_x$ into itself, 
with operator norm $\| L \|_{\mL(H^s_x)} := \sup \{ \| Lh \|_s : h \in H^s_x, \| h \|_s = 1 \}$.  

\begin{lemma}  \label{lemma:App2}
$(i)$ \emph{(LWP).} 
Let $T  > 0$, $s \in \R$, $\mR \in C([0,T], \mL(H^s_x))$, and let
\begin{equation} \label{App3.1}
r_s := \| \mR \|_{C([0,T], \mL(H^s_x))} = \sup_{t \in [0,T]} \| \mR(t) \|_{\mL(H^s_x)},
\qquad  \mL_5:= \pa_t + m \pa_{xxx} + \mR\ .
\end{equation}
Let $\a \in H^s_x$ and $f \in C([0,T], H^s_x)$. 
If $T \, r_s \leq 1/2$, then the Cauchy problem
\begin{equation} \label{App4}
\begin{cases}
\mL_5 u = f \\ 
u(0,x) = \a(x)
\end{cases}
\end{equation}
has a unique solution $u \in C([0,T], H^s_x)$.
The solution $u$ satisfies 
\begin{equation} \label{App25}
\| u \|_{T,s} 
\leq (1 + 2 T r_s) \| \a \|_s + 2 T \| f \|_{T,s}
\leq 2 (\| \a \|_s + T \| f \|_{T,s}).
\end{equation}

\medskip

$(ii)$ \emph{(Tame LWP).} 
Let $T > 0$, $s \in \R$, $s_1 \in \R$ with $s \geq s_1$, 
and let $\mR \in C([0,T], \mL(H^s_x))$ $\cap$ $C([0,T], \mL(H^{s_1}_x))$. 
Assume that 
\begin{equation} \label{App4.1}
\| \mR(t) h \|_s \leq c_1 \| h \|_s + c_s \| h \|_{s_1} \,, \quad 
\| \mR(t) h \|_{s_1} \leq c_1 \| h \|_{s_1} \quad \forall h \in H^s_x \,,
\end{equation}
for all $t \in [0,T]$, where $c_1, c_s$ are positive constants.  
Let $\a \in H^s_x$. 
If 
\begin{equation}  \label{Tc1 small}
T c_1 \leq 1/2, 
\end{equation}
then the solution $u \in C([0,T],H^{s_1}_x)$ of the Cauchy problem \eqref{App4} given in $(i)$ 
belongs to $C([0,T], H^s_x)$, with
\begin{equation} \label{App4.2}
\| u \|_{T,s} \leq 2 T \| f \|_{T,s} 
+ (1 + 2Tc_1) \| \a \|_s 
+ 4 T c_s (T \| f \|_{T,s_1} + \| \a \|_{s_1}).
\end{equation}

\medskip

$(iii)$ \emph{(GWP).}
Let $T > 0$, $s \in \R$, $\mR \in C([0,T], \mL(H^s_x))$, and let $r_s$ be defined in \eqref{App3.1}. 
Let $\a \in H^s_x$. 
Then the Cauchy problem \eqref{App4} has a unique global solution $u \in C([0,T], H^s_x)$, with 
\begin{equation} \label{App19}
\| u \|_{T,s} \leq 2^{4 T r_s} (\| \a \|_s + 4 T \| f \|_{T,s}).
\end{equation}

$(iv)$ \emph{(Tame GWP).} 
Let $T > 0$, $s \in \R$, $s_1 \in \R$ with $s \geq s_1$, 
and let $\mR \in C([0,T], \mL(H^s_x))$ $\cap$ $C([0,T], \mL(H^{s_1}_x))$. 
Assume that \eqref{App4.1} holds for all $t \in [0,T]$, where $c_1, c_s$ are positive constants. 
Let $\a \in H^s_x$. 
Then the global solution $u \in C([0,T],H^s_x)$ of the Cauchy problem \eqref{App4} 
given in $(iii)$ satisfies 
\begin{equation} \label{App20}
\| u \|_{T,s} \leq 2^{4 T c_1} ( \| \a \|_s + 4 T c_s \| \a \|_{s_1}  
+ 2 T \| f \|_{T,s} + 4 T^2 c_s \| f \|_{T,{s_1}}).
\end{equation}
\end{lemma}

\begin{proof}
$(i)$ Write $u = v + w$, where $v(t,x)$ is the solution of 
\begin{equation}  \label{App5}
\pa_t v + m \pa_{xxx} v = 0,  \quad v(0,x) = \a(x). 
\end{equation}
Hence $u$ solves \eqref{App4} if and only if $w(t,x)$ solves 
\begin{equation}  \label{App6}
\pa_t w + m \pa_{xxx} w + \mR w = - \mR v + f, \quad 
w(0,x) = 0.
\end{equation}
By Lemma \ref{lemma:App1}, \eqref{App6} is the fixed point problem 
\begin{equation}  \label{App7}
w = \Psi(w), 
\end{equation}
where $\Psi(w) := A [f - \mR (v+w)]$.
Let $B_\rho := \{ w \in C([0,T], H^s_x) : \| u \|_{T,s} \leq \rho \}$, $\rho \geq 0$. 
Then 
\begin{equation}  \label{App8}
\| \Psi(w) \|_{T,s} \leq T \, (\| f \|_{T,s} + r_s \| \a \|_s + r_s \rho), \qquad 
\| \Psi(w_1) - \Psi(w_2) \|_{T,s} \leq T \, r_s \| w_1 - w_2 \|_{T,s} 
\end{equation}
for all $w, w_1, w_2 \in B_\rho$. 
By assumption, $T \, r_s \leq 1/2$. 
Therefore, for any $\rho \geq 2 T (\| f \|_{T,s} + r_s \| \a \|_s)$, 
$\Psi$ is a contraction in $B_\rho$.
In particular, we fix $\rho = \rho_0 := 2 T (\| f \|_{T,s} + r_s \| \a \|_s)$.
Hence there exists a fixed point $w \in B_{\rho_0}$ of $\Psi$, with 
$\| w \|_{T,s} \leq \rho_0 
\leq 2T \| f \|_{T,s} + \| \a \|_s$. 
As a consequence, there exists a solution $u \in C([0,T],H^s_x)$ of \eqref{App4} with 
$\| u \|_{T,s} \leq 2 (T \| f \|_{T,s} + \| \a \|_s)$. 
By the contraction lemma, the solution $u$ is unique in any ball $B_\rho$, $\rho \geq \rho_0$, 
and therefore it is unique in $C([0,T],H^s_x)$. 

\medskip

$(ii)$ By assumption, $T c_1 \leq 1/2$, and therefore, by $(i)$, there exists a unique solution 
$u \in C([0,T],H^{s_1}_x)$. It remains to prove that $u$ satisfies \eqref{App4.2}. 
By construction, $u = v + w$, where $v \in C([0,T],H^s_x)$ is the solution of \eqref{App5}, 
with $\| v(t) \|_s = \| \a \|_s$ for all $t \in [0,T]$, 
and $w \in C([0,T], H^{s_1}_x)$ solves \eqref{App7}. 
By the iterative scheme of the contraction lemma, $w$ is the limit in $C([0,T],H^{s_1}_x)$ of the sequence $(w_n)$, where $w_0 := 0$, and $w_{n+1} := \Psi(w_n)$ for all $n \in \N$. 
By \eqref{App4.1} and \eqref{App3}, $\Psi$ maps $C([0,T], H^s_x)$ into itself, 
therefore $w_n \in C([0,T], H^s_x)$ for all $n \geq 0$. 
Let $h_n := w_n - w_{n-1}$, $n \geq 1$, so that $w_n = \sum_{k=1}^n h_k$.
One has $h_{n+1} = - A \mR h_n$ for all $n \geq 1$, and 
\[
\| h_{n+1} \|_{T,s} \leq T c_1 \| h_n \|_{T,s} + T c_s \| h_n \|_{T,s_1}, \quad 
\| h_{n+1} \|_{T,s_1} \leq T c_1 \| h_n \|_{T,s_1}, \quad \forall n \geq 1.
\]
Hence, by induction, for all $n \geq 1$ we have
\begin{equation}  \label{App11}
\begin{aligned}
\| h_n \|_{T,s} & \leq (T c_1)^{n-1} \| h_1 \|_{T,s} 
+ (n-1) (T c_1)^{n-2} T c_s \| h_1 \|_{T,s_1}, \\
\| h_n \|_{T,s_1} & \leq (T c_1)^{n-1} \| h_1 \|_{T,s_1}.
\end{aligned}
\end{equation}
Also, $\| h_1 \|_{T,s} \leq T \| f \|_{T,s} + T c_1 \| \a \|_s + T c_s \| \a \|_{s_1}$ 
and $\| h_1 \|_{T,s_1} \leq T \| f \|_{T,s_1} + T c_1 \| \a \|_{s_1}$. 
Therefore
\begin{align}  \notag
\| h_n \|_{T,s} 
& \leq (T c_1)^{n-1} T \| f \|_{T,s} 
+ (T c_1)^n \| \a \|_s 
+ (n-1) (T c_1)^{n-2} T c_s T \| f \|_{T,s_1} 
\\ & \quad \notag 
+ n (T c_1)^{n-1} T c_s \| \a \|_{s_1}, 
\\
\| h_n \|_{T,s_1} 
& \leq (T c_1)^{n-1} T \| f \|_{T,s_1} 
+ (T c_1)^n \| \a \|_{s_1} \qquad \forall n \geq 1.
\label{App11.2}
\end{align}
Since $T c_1 \leq 1/2$, the sequence $w_n = \sum_{k=1}^n h_k$ converges in $C([0,T],H^s_x)$
to some limit $\tilde w \in C([0,T],H^s_x)$. Since $w_n$ converges to $w$ in $C([0,T],H^{s_1}_x)$, 
the two limits coincide, and $w \in C([0,T],H^s_x)$. 
Since $\| w \|_{T,s} \leq \sum_{k=1}^\infty \| h_k \|_{T,s}$, 
we get 
\begin{equation}  \label{App12}
\| w \|_{T,s} \leq 2 T (\| f \|_{T,s} + c_1 \| \a \|_s)
+ 4 T c_s (T \| f \|_{T,s_1} + \| \a \|_{s_1}). 
\end{equation}
Since $u = v+w$, we deduce \eqref{App4.2}. 

\medskip

$(iii)$. If $T r_s \leq 1/2$, the result is given by $(i)$. 
Let $T r_s > 1/2$, and fix $N \in \N$ such that $2 T r_s \leq N \leq 4 T r_s$. 
Let $T_0 := T/N$, so that $1/4 \leq T_0 r_s \leq 1/2$. 
Divide the interval $[0,T]$ in the union $I_1 \cup \ldots \cup I_N$, where 
$I_n := [(n-1)T_0, nT_0]$. 
Applying $(i)$ on the time interval $I_1 = [0,T_0]$ 
gives the solution $u_1 \in C(I_1, H^s_x)$, 
with $\| u_1 \|_{C(I_1,H^s_x)} \leq b \| \a \|_s + 2T_0 \| f \|_{T,s}$, where 
$b := 1 + 2 T_0 r_s$. 
Now consider the Cauchy problem on $I_2$ with initial datum $u(T_0) = u_1(T_0)$. 
Applying $(i)$ on $I_2$ gives the solution $u_2 \in C(I_2, H^s_x)$, with 
\[
\| u_2 \|_{C(I_2, H^s_x)} 
\leq b \| u_1(T_0) \|_s + 2 T_0 \| f \|_{T,s}
\leq b^2 \| \a \|_s + (1+b) 2 T_0 \| f \|_{T,s}.
\]
We iterate the procedure $N$ times. At the last step, we find the solution $u_N$ defined on $I_N$, with $\| u_N \|_{C(I_N, H^s_x)} \leq b^N \| \a \|_s + (b^N-1) \frac{1}{b-1} 
2 T_0 \| f \|_{T,s}$.
We define $u(t) := u_n(t)$ for $t \in I_n$, and the thesis follows, using that $b \leq 2$. 

\medskip

$(iv)$ If $T c_1 \leq 1/2$, the result is given by $(ii)$. 
Let $T c_1 > 1/2$, and fix $N \in \N$ such that $2 T c_1 \leq N \leq 4 T c_1$. 
Let $T_0 := T/N$, so that $1/4 \leq T_0 c_1 \leq 1/2$. 
Split $[0,T] = I_1 \cup \ldots \cup I_N$, where $I_n := [(n-1)T_0, nT_0]$. 
Perform the same procedure as above. Using \eqref{App4.2}, and $1 + 2 T_0 c_1 \leq 2$, 
by induction we get 
\begin{align*} 
\| u_n \|_{C(I_n, H^s_x)} 
& \leq 2^n \| \a \|_s 
+ (2^n -1) 2 T_0 \| f \|_{T,s} 
+ n 2^{n-1} 4 T_0 c_s \| \a \|_{s_1} 
\\ & \quad 
+ [2^n(n-1) + 1] 4 T_0 c_s T_0 \| f \|_{T,s_1} \,,
\\
\| u_n \|_{C(I_n, H^{s_1}_x)} 
& \leq 2^n \| \a \|_{s_1}
+ (2^n-1) 2 T_0 \| f \|_{T,s_1}.
\end{align*}
This implies \eqref{App20}, recalling that $T_0 c_1 \leq 1/2$ and also $N T_0 = T$, $N \geq 1$.
\end{proof}

\begin{lemma}\label{lemma:App3}
There exist universal positive constants $\s, \d_*$ with the following properties.
Let $s \geq 0$, let $m\geq 1/2$, and let $a_{14}(t,x), a_{15}(t,x)$ be two functions with 
$a_{14}, \pa_t a_{14}, a_{15} \in C([0,T],H^{s+\s}_x)$ and $\int_\T a_{14} (t,x) \, dx = 0$,
and let $\mL_4:=\pa_t + m \pa_{xxx} + a_{14} \pa_x + a_{15}$.
Let 
\[
\d(\mu) := \| a_{14}, \pa_t a_{14}, a_{15} \|_{T,\mu + \s} 
\quad \forall \mu \in [0,s].
\]
Assume $\d(0)\leq \d_*$.
Let $f\in C([0,T],H^s_x)$, $\a\in H^s_x$.
Then the Cauchy problem
\begin{equation} \label{AppL4}
\ttL_4 u = f, \quad 
u(0) = \a
%\begin{cases}
%\ttL_4 u = f \\ 
%u(0) = \a
%\end{cases}
\end{equation}
admits a unique solution $u \in C([0,T], H^s_x)$, with 
\begin{equation} \label{en.est.30}
\| u \|_{T,s} \leq C_s \left\{\| f \|_{T,s} + \| \a \|_s + \d (s) (\| f \|_{T,0} + \| \a \|_0 ) \right\}\ .
\end{equation}
\end{lemma}

\begin{proof}
Following the procedure given in Section \ref{step-5}, we define $\mS := I + \g(t,x) \pa_x^{-1}$ (see \eqref{L30}) with $\g(t,x):=-\tfrac{1}{3m}\pa_x^{-1}a_{14}(t,x)$.
We have that $u$ solves \eqref{AppL4}
if and only if $\widetilde u := \mS^{-1} u$ satisfies
\[
\ttL_5 \widetilde u = \widetilde f, \quad  
\widetilde u(0) = \widetilde \a
%\begin{cases} 
%\ttL_5 \widetilde u = \widetilde f \\ 
%\widetilde u(0) = \widetilde \a
%\end{cases}
\]
where $\widetilde f := \mS^{-1} f$, $\widetilde \a := \mS^{-1}(0) \a$ and $\ttL_5=\pa_t+m\pa_{xxx}+\mR$, with $\mR=\mS^{-1}\{a_{15}+(a_{14}\g-(a_{14})_x)\pi_0 + (\ttL_4 \g) \pa_x^{-1}\}$.
Then the thesis follows by Lemmas \ref{lemma:App2} and \ref{lemma:mS}.
%as a straightforward consequence of Lemma \ref{lemma:App2} and Lemma \ref{lemma:mS}.
\end{proof}

\begin{lemma}\label{lemma:App4}
There exist universal positive constants $\s, \d_*$ with the following properties.
Let $s \geq 0$, let $m\geq 1/2$, and let $a_{12}(t,x), a_{13}(t,x)$ be two functions with 
$a_{12}, \pa_t a_{12}, a_{13} \in C([0,T],H^{s+\s}_x)$,
and let $\mL_3:=\pa_t + m \pa_{xxx} + a_{12} \pa_x + a_{13}$.
Let 
\[ 
\d(\mu) := \| a_{12}, \pa_t a_{12}, a_{13} \|_{T,\mu + \s} 
\quad \forall \mu \in [0,s].
\]
Assume $\d(0)\leq \d_*$.
Let $f\in C([0,T],H^s_x)$, $\a\in H^s_x$.
Then the Cauchy problem
\begin{equation} \label{AppL3}
\ttL_3 u = f, \quad 
u(0) = \a
%\begin{cases}
%\ttL_3 u = f \\ 
%u(0) = \a
%\end{cases}
\end{equation}
admits a unique solution $u \in C([0,T], H^s_x)$, with 
\begin{equation} \label{en.est.31}
\| u \|_{T,s} \leq C_s \left\{\| f \|_{T,s} + \| \a \|_s + \d (s) (\| f \|_{T,0} + \| \a \|_0 ) \right\}\ .
\end{equation}
\end{lemma}

\begin{proof}
Following the procedure given in Section \ref{step-4}, we define $\mT h (t,x) := h (t, x + p(t))$ (see \eqref{L25}) with $p(t):=-\tfrac{1}{2\pi}\int_0^t \int_\T a_{12}(s,x) \; dx ds$.
We have that $u$ solves \eqref{AppL3}
if and only if $\widetilde u := \mT^{-1} u$ satisfies
\[
\ttL_4 \widetilde u = \widetilde f, \quad  
\widetilde u(0) = \a
%\begin{cases} 
%\ttL_4 \widetilde u = \widetilde f \\ 
%\widetilde u(0) = \a
%\end{cases}
\]
(note that $\mT(0)$ is the identity) where $\widetilde f := \mT^{-1} f$, and $\ttL_4=\pa_t+m\pa_{xxx}+ a_{14} \pa_x + a_{15}$, with $a_{14},a_{15}$ given by formula \eqref{L27}.
Then the thesis follows by Lemmas \ref{lemma:App3} and \ref{lemma:mT}.
%as a straightforward consequence of Lemma \ref{lemma:App3} and Lemma \ref{lemma:mT}.
\end{proof}

\begin{lemma}\label{lemma:App5}
There exist universal positive constants $\s, \d_*$ with the following properties.
Let $s \geq 0$, let $m\geq 1/2$, and let $a_8(t,x), a_9(t,x),a_{10}(t,x)$ be three functions with 
$a_8, \pa_t a_8, a_9$, $\pa_t a_9, a_{10} \in C([0,T],H^{s+\s}_x)$ and $\int_\T a_8(t,x)\; dx = 0$,
and let $\mL_2:=\pa_t + m \pa_{xxx} + a_8 \pa_{xx} + a_9 \pa_x + a_{10}$.
Let 
\[ 
\d(\mu) := \| a_8, \pa_t a_8, a_9, \pa_t a_9, a_{10} \|_{T,\mu + \s} 
\quad \forall \mu \in [0,s].
\] 
Assume $\d(0)\leq \d_*$.
Let $f\in C([0,T],H^s_x)$, $\a\in H^s_x$.
Then the Cauchy problem
\begin{equation} \label{AppL2}
\ttL_2 u = f, \quad u(0) = \a
%\begin{cases}
%\ttL_2 u = f \\ 
%u(0) = \a
%\end{cases}
\end{equation}
admits a unique solution $u \in C([0,T], H^s_x)$, with 
\begin{equation} \label{en.est.32}
\| u \|_{T,s} \leq 
C_s \left\{\| f \|_{T,s} + \| \a \|_s + \d (s) (\| f \|_{T,0} + \| \a \|_0 ) \right\}.
\end{equation}
\end{lemma}

\begin{proof}
Following the procedure given in Section \ref{step-3}, we define $\mM h (t,x) := q(t,x) h(t,x)$ (see \eqref{M1}) with $q(t,x):=\exp\{-\tfrac{1}{3m} (\pa_x^{-1} a_8) (t,x) \}$.
We have that $u$ solves \eqref{AppL2}
if and only if $\widetilde u := \mM^{-1} u$ satisfies
\[
\ttL_3 \widetilde u = \widetilde f, \quad 
\widetilde u(0) = \widetilde \a
%\begin{cases} 
%\ttL_3 \widetilde u = \widetilde f \\ 
%\widetilde u(0) = \widetilde \a
%\end{cases}
\]
where $\widetilde f := \mM^{-1} f$, $\widetilde \a:= \mM^{-1}(0) \a$, and $\ttL_3=\pa_t+m\pa_{xxx}+ a_{12} \pa_x + a_{13}$, with $a_{12},a_{13}$ given by formula \eqref{M3}.
Then the thesis follows by Lemmas \ref{lemma:App4} and \ref{lemma:mM}.
%as a straightforward consequence of Lemma \ref{lemma:App4} and Lemma \ref{lemma:mM}.
\end{proof}

\begin{lemma}\label{lemma:App6}
There exist universal positive constants $\s, \d_*$ with the following properties.
Let $s \geq 0$ and let $a_4(t), a_5(t,x),a_6(t,x),a_7(t,x)$ be four functions with $a_4 \in C^1 ([0,T], \R)$, $a_5, \pa_t a_5, a_6, \pa_t a_6, a_7 \in C([0,T],H^{s+\s}_x)$ and $\int_\T a_5(t,x)\; dx = 0$, and let $\mL_1:=\pa_t + a_4 \pa_{xxx} + a_5 \pa_{xx} + a_6 \pa_x + a_7$.
Let 
\begin{equation}\label{m.17}
\d(\mu) := \sup_{t\in[0,T]} | a_4(t) -1 | + \sup_{t\in(0,T)} |a_4' (t)| + \| a_5, \pa_t a_5, a_6, \pa_t a_6, a_7 \|_{T,\mu + \s} 
\quad \forall \mu \in [0,s].
\end{equation}
Assume $\d(0)\leq \d_*$.
Let $f\in C([0,T],H^s_x)$, $\a\in H^s_x$.
Then the Cauchy problem
\begin{equation} \label{AppL1}
\ttL_1 u = f, \quad u(0) = \a
%\begin{cases}
%\ttL_1 u = f \\ 
%u(0) = \a
%\end{cases}
\end{equation}
admits a unique solution $u \in C([0,T], H^s_x)$, with 
\begin{equation} \label{en.est.33}
\| u \|_{T,s} \leq C_s \left\{\| f \|_{T,s} + \| \a \|_s + \d (s) (\| f \|_{T,0} + \| \a \|_0 ) \right\}.
\end{equation}
\end{lemma}

\begin{proof}
Following the procedure given in Section \ref{step-2}, we define $\mB h (t,x) := h( \psi(t) ,x)$ (see \eqref{L17}) with $\psi(t):=\tfrac{1}{m}\int_0^t a_4(s)\, ds$, 
where $m:=\tfrac{1}{T} \int_0^T a_4(t) \, dt$.
We have that $u$ solves \eqref{AppL1}
if and only if $\widetilde u := \mB^{-1} u$ satisfies
\[
\ttL_2 \widetilde u = \widetilde f, \quad  
\widetilde u(0) = \a
%\begin{cases} 
%\ttL_2 \widetilde u = \widetilde f \\ 
%\widetilde u(0) = \a
%\end{cases}
\]
(note that $\mB(0)$ is the identity) where $\widetilde f := \mB^{-1} f$, and $\ttL_2=\pa_t+m\pa_{xxx} + a_8 \pa_{xx} + a_9 \pa_x + a_{10}$, with $a_8, a_9,  a_{10}$ given by formula \eqref{L24} (see also \eqref{L18}).
Then the thesis follows by Lemma \ref{lemma:App5} and \ref{lemma:mB}.
%as a straightforward consequence of Lemma \ref{lemma:App5} and Lemma \ref{lemma:mB}.
\end{proof}

\begin{lemma}\label{lemma:App7}
There exist universal positive constants $\s, \d_*$ with the following properties.
Let $s \geq 0$ and let $a_3(t,x), a_2(t,x), a_1(t,x), a_0(t,x)$ be four functions with 
$a_3$, $\pa_t a_3$, $\pa_{tt} a_3$, $a_1$, $\pa_t a_1$, $a_0 \in C([0,T],H^{s+\s}_x)$ and $a_2 = c \pa_x a_3$ for some $c\in\R$. Let 
\begin{equation}\label{m.18}
\d(\mu) := \| a_3, \pa_t a_3, \pa_{tt} a_3, a_1, \pa_t a_1, a_0 \|_{T,\mu + \s} 
\quad \forall \mu \in [0,s].
\end{equation}
Assume $\d(0)\leq \d_*$.
Let $\mL_0:=\pa_t + (1+a_3) \pa_{xxx} + a_2 \pa_{xx} + a_1 \pa_x + a_0$.
Let $f\in C([0,T],H^s_x)$, $\a\in H^s_x$.
Then the Cauchy problem
\begin{equation} \label{AppL0}
\ttL_0 u = f, \quad u(0) = \a
%\begin{cases}
%\ttL_0 u = f \\ 
%u(0) = \a
%\end{cases}
\end{equation}
admits a unique solution $u \in C([0,T], H^s_x)$, with 
\begin{equation} \label{en.est.34}
\| u \|_{T,s} 
\leq C_s \left\{\| f \|_{T,s} + \| \a \|_s + \d (s) (\| f \|_{T,0} + \| \a \|_0 ) \right\}.
\end{equation}
\end{lemma}

\begin{proof}
Following the procedure given in Section \ref{step-1}, we define $(\mA h) (t,x) := h( t ,x + \b(t,x))$ (see \eqref{L4}) with $\b(t,x):= (\pa_x^{-1} \rho_0) (t,x)$, where 
$\rho_0$ is defined in \eqref{L12}-\eqref{L13}.
We have that $u$ solves \eqref{AppL0}
if and only if $\widetilde u := \mA^{-1} u$ satisfies
\[
\ttL_1 \widetilde u = \widetilde f, \quad  
\widetilde u(0) = \widetilde \a
%\begin{cases} 
%\ttL_1 \widetilde u = \widetilde f \\ 
%\widetilde u(0) = \widetilde \a
%\end{cases}
\]
where $\widetilde f := \mA^{-1} f$, $\widetilde \a := \mA^{-1}(0) \a$, and $\ttL_1=\pa_t+a_4\pa_{xxx} + a_5 \pa_{xx} + a_6 \pa_x + a_7$, with $a_4$ not depending on the space variable $x$ and with $a_4, a_5,  a_6,a_7$ given by formula \eqref{L8}.
Then the thesis follows by Lemmas \ref{lemma:App6} and \ref{lemma:mA}.
%as a straightforward consequence of Lemma \ref{lemma:App6} and Lemma \ref{lemma:mA}.
\end{proof}

\begin{remark}
\label{rem:Lk}
\let\qed\relax
Consider the operators $\mL_0$, \ldots, $\mL_5$ 
defined in Lemmas \ref{lemma:App2}-\ref{lemma:App7}. 
Define 
\begin{align*}
\mL_0^* h & := -\pa_t h - \pa_{xxx} [(1+a_3) h] + \pa_{xx} (a_2 h) -\pa_x (a_1 h) + a_0 h\\
\mL_1^* h & := -\pa_t h - a_4 \pa_{xxx}  h + \pa_{xx} (a_5 h) -\pa_x (a_6 h) + a_7 h\\
\mL_2^* h & := -\pa_t h - m \pa_{xxx}  h + \pa_{xx} (a_8 h) -\pa_x (a_9 h) + a_{10} h\\
\mL_3^* h & := -\pa_t h - m \pa_{xxx}  h -\pa_x (a_{12} h) + a_{13} h\\
\mL_4^* h & := -\pa_t h - m \pa_{xxx}  h -\pa_x (a_{14} h) + a_{15} h\\
\mL_5^* h & := -\pa_t h - m \pa_{xxx}  h + \mR^T h.
\end{align*}
It is straightforward to check that Lemmas \ref{lemma:App2}-\ref{lemma:App7} also hold when the operator $\mL_k$ ($k=0, \ldots , 5$) is replaced by $\mL_k^*$. 
The crucial observation is that for all $k=0, \ldots , 5$ (see Remark \ref{come fosse Ham} for the case $k=0$) the operator $-\mL_k^*$ has the same structure as $\mL_k$ (one might need to worsen the constants $\sigma$ since the coefficients of $-\mL_k^*$ involve space derivatives of the coefficients of $\mL_k$). 
It is also immediate to verify that the same estimates also hold for the backward Cauchy problems
\begin{equation} \label{App13}
\begin{cases}
\mL_k u = f \\ 
u(T) = \a
\end{cases}
\qquad\qquad
\begin{cases}
\mL_k^* u = f \\ 
u(T) = \a
\end{cases}
\qquad \qquad k=0,\ldots,5. 
\end{equation}
\end{remark}

\section{Appendix B. Nash-Moser theorem} 
\label{sec:NM}

In this section we prove a Nash-Moser implicit function theorem 
that is a modified version of the theorem in H\"ormander \cite{Olli}.
With respect to \cite{Olli}, here (Theorem \ref{thm:NM}) 
we assume slightly stronger hypotheses on the nonlinear operator $\Phi$ 
and its second derivative. 
These hypotheses are naturally verified in applications to PDEs.
We use the iteration scheme of \cite{Geodesy} 
(called \emph{discrete Nash method} by H\"ormander),
which is neither the Newton scheme with smoothings 
used in \cite{BBPro}, \cite{BCPro}, \cite{BBM-auto}, 
nor the scheme in \cite{Olli} and \cite{AG}. 
The scheme of \cite{Geodesy} is based on a telescoping series like in \cite{Olli}, 
but some corrections $y_n$ (see \eqref{Olli.8}) are also introduced.
In this way the scheme converges directly to a solution of the equation 
$\Phi(u) = \Phi(0) + g$, avoiding the intermediate step in \cite{Olli} 
where Leray-Schauder theorem is applied. 
This makes it possible to remove two assumptions of H\"ormander's theorem 
\cite{Olli}, which are 
the compact embeddings $F_b \hookrightarrow F_a$ in the codomain scale 
of Banach spaces $(F_a)_{a \geq 0}$, 
and the continuity of the approximate right inverse $\Psi(v)$ 
with respect to the approximate linearization point $v$.
We point out that, unlike Theorem 2.2.2 of \cite{Geodesy}, 
our Theorem \ref{thm:NM} also applies to the case of Sobolev spaces.

Let us begin with recalling the construction of ``weak'' spaces in \cite{Olli}.

\medskip

Let $E_a$, $a \geq 0$, be a decreasing family of Banach spaces with injections  
$E_b \hookrightarrow E_a$ of norm $\leq 1$ when $b \geq a$. 
Set $E_\infty = \cap_{a\geq 0} E_a$ with the weakest topology making the 
injections $E_\infty \hookrightarrow E_a$ continuous. 
Assume that $S_\theta : E_0 \to E_\infty$ for $\theta \geq 1$ are linear operators 
such that, with constants $C$ bounded when $a$ and $b$ are bounded,
\begin{alignat}{2}
\label{S1} \| S_{\theta} u \|_b & \leq C \|u\|_a & & \text{ if }\; b\leq a; \\
\label{S2} \| S_{\theta} u\|_b & \leq C\theta^{b-a} \|u\|_a & & \text{ if }\; a<b; \\
\label{S3} \| u-S_{\theta}u\|_b & \leq C\theta^{b-a}\|u\|_a & & \text{ if }\; a>b; \\ 
\label{S4} \Big\|\frac{d}{d\theta} S_{\theta} u \Big\|_b & \leq C\theta^{b-a-1}\|u\|_a \,. & &
\intertext{
From \eqref{S2}-\eqref{S3} one can obtain the logarithmic convexity of the norms}
\label{S5} \|u\|_{\lambda a +(1-\lambda) b } & \leq C\|u\|_a^\lambda\|u\|_b^{1-\lambda} 
&& \text{ if } \; 0<\lambda<1.
\end{alignat}
Consider the sequence $\{\theta_j\}_{j\in\N},$ with $1=\theta_0 < \theta_1 < \ldots \rightarrow \infty$, such that $\frac{\theta_{j+1}}{\theta_j}$ is bounded. 
Set $\Delta_{j} := \theta_{j+1}-\theta_j$ and 
\begin{equation} \label{new.24}
R_0 u := \frac{S_{\theta_1}u}{\Delta_0}\,, \qquad 
R_j u := \frac{S_{\theta_{j+1}} u - S_{\theta_j}u}{\Delta_j}\,, \quad  j \geq 1.
\end{equation}
By \eqref{S3} we deduce that, if $u\in E_b$ for some $b>a$, then 
\begin{equation} \label{u serie} 
u = \sum_{j=0}^\infty \Delta_j R_j u
\end{equation}
with convergence in $E_a$.
Moreover, \eqref{S4} implies that, for all $b$, 
\begin{equation}\label{Rj}
\|R_ju\|_b\leq C_{a,b}\theta_j^{b-a-1} \| u \|_a \,. 
\end{equation}
Conversely, assume that $a_1<a<a_2$, that $u_j \in E_{a_2}$ and that 
\begin{equation}\label{u_j}
\|u_j\|_b\leq M \theta_j^{b-a-1} \quad \text{if} \ \ b=a_1 \ \ \text{or} \ \ b=a_2.
\end{equation}
By \eqref{S5} this remains true with a constant factor on the right-hand side if $a_1<b<a_2$, 
so that $u=\sum\Delta_j u_j$ converges in $E_b$ if $b<a$. 

Let $E'_a$ be the set of all sums $u=\sum\Delta_j u_j$ 
with $u_j$ satisfying \eqref{u_j} and introduce the norm 
$\|u\|'_a$ as the infimum of $M$ over all such decompositions.
It follows that $\| \ \|^\prime_a$ is stronger than $\| \ \|_b$ if $a > b$, 
while \eqref{u serie} and \eqref{Rj} show that $\| \ \|^\prime_a$ is weaker than $\| \ \|_a$.
Moreover 
$(i)$ the space $E'_a$ and, up to equivalence, its norm are independent of the choice of $a_1$ and $a_2$; 
$(ii)$ $E'_a$ is defined by \eqref{Rj} for any values of $b$ to the left and to the right of $a$;
$(iii)$ $E'_a$ does not depend on the smoothing operators;
$(iv)$ in \eqref{S3} we can replace $\|u\|_a$ by $\|u\|_a'$, namely 
\begin{equation} \label{S3 plus} 
\| u - S_\theta u \|_b 
\leq C'_{a,b} \theta^{b-a} \| u \|_a' \quad \text{if} \ \ a > b,  
\end{equation}
if we take another constant $C'_{a,b}$, 
which may tend to $\infty$ as $b$ approaches $a$. 
All these four statements $(i)$-$(iv)$ are proved in \cite{Olli}.

%\begin{definition}
%Fix $a>0$. If $(u_k)$ is a bounded sequence in $E_a$, 
%and $u_k\rightarrow u$ in $E_0$, 
%then it follows from \eqref{S5} that $(u_k)$ is a Cauchy sequence in $E_b$ 
%for every $b<a$, so that the limit $u\in E_b$. 
%In this case we say that $(u_k) \subseteq E_a$ is \emph{weakly convergent}, 
%and that $u$ is the \emph{weak $E_a$ limit} of $(u_k)$.
%\end{definition}
%
%By the definition of $E'_a$ one can show that every element in $E'_a$ is the weak $E_a$ limit of a sequence in $E_\infty.$

Now let us suppose that we have another family $F_a$ of decreasing Banach spaces with smoothing operators having the same properties as above. We use the same notation also for the smoothing operators. 
Unlike \cite{Olli}, here we do not need to assume that the embedding 
$F_b \hookrightarrow F_a$ is compact for $b>a$.

\begin{theorem} \label{thm:NM}
Let $a_1, a_2, \a, \b, a_0, \mu$ be real numbers with 
\begin{equation} \label{ineq 2016}
0 \leq a_0 \leq \mu \leq a_1, \quad 
a_1 + \frac{\b}{2} \, \leq \a < a_1 + \b \leq a_2, \quad 
2\a < a_1 + a_2. 
\end{equation}
Let $V$ be a convex neighborhood of $0$ in $E_\mu$. 
Let $\Phi$ be a map from $V$ to $F_0$ such that $\Phi : V \cap E_{a+\mu} \to F_a$ 
is of class $C^2$ for all $a \in [0, a_2 - \mu]$, with 
\begin{equation}\label{Phi sec}
\|\Phi''(u)[v,w] \|_a \leq C \big( \| v \|_{a+\mu} \| w \|_{a_0} 
+ \| v \|_{a_0} \| w \|_{a+\mu}
+ \| u \|_{a+\mu} \| v \|_{a_0} \| w \|_{a_0} \big)
\end{equation}
for all $u \in V \cap E_{a+\mu}$, $v,w \in E_{a+\mu}$.
Also assume that $\Phi'(v)$, for $v \in E_\infty \cap V$ 
belonging to some ball $\| v \|_{a_1} \leq \d_1$,
has a right inverse $\Psi(v)$ mapping $F_\infty$ to $E_{a_2}$, and that
\begin{equation}  \label{tame in NM}
\|\Psi(v)g\|_a\leq C(\|g\|_{a + \b - \a} + \| g \|_0 \| v \|_{a + \b}) 
\quad \forall a \in [a_1, a_2].
\end{equation}
There exists $\d > 0$ such that, for every $g \in F'_\b$ in the ball $\| g \|_\b' \leq \d$,
there exists $u \in E_\a'$, with $\| u \|_\a' \leq C \| g \|_\b'$, 
solving $\Phi(u) = \Phi(0) + g$.
\end{theorem}

\begin{proof}
We follow the proof in \cite{Olli} where possible, 
but we use a different iteration scheme. 
Let $\theta_j := j+1$, 
so that $\D_j = 1$ for all $j$. 
Let $g \in F_\b'$ and $g_j := R_j g$. Thus 
\begin{equation} \label{Olli.7}
g = \sum_{j=0}^\infty g_j, \quad
\| g_j \|_b \leq C_b \theta_j^{b-\b-1} \| g \|_\b'
\quad \forall b \in [0,+\infty). 
\end{equation}
We claim that if $\| g \|_\b'$ is small enough, 
then we can define a sequence $u_j \in V \cap E_{a_2}$ with $u_0 := 0$ by the recursion formula
\begin{equation} \label{Olli.8}
u_{j+1} := u_j + h_j, \quad
v_j := S_{\theta_j} u_j, \quad
h_j := \Psi(v_j) (g_j + y_j) \quad \forall j \geq 0,
\end{equation}
where $y_0 := 0$, 
\begin{equation} \label{new.1}
y_1 := - S_{\theta_1} e_0, \qquad
y_j := - S_{\theta_j} e_{j-1} - R_{j-1} \sum_{i=0}^{j-2} e_i 
\quad \ \forall j \geq 2,
\end{equation}
and $e_j := e_j' + e_j''$, 
\begin{equation} \label{new.2}
e_j' := \Phi(u_j + h_j) - \Phi(u_j) - \Phi'(u_j) h_j , 
\qquad
e_j'' := (\Phi'(u_j) - \Phi'(v_j)) h_j.
\end{equation}
We prove that for all $j \geq 0$
\begin{alignat}{2}
\| h_j \|_a 
& \leq K_1 \| g \|_\b' \, \theta_j^{a-\a-1} \quad && \forall a \in [a_1, a_2], 
\label{Olli.9}
\vspace{2pt} \\
\| v_j \|_a 
& \leq K_2 \| g \|_\b' \, \theta_j^{a-\a} \quad && \forall a \in [a_1 + \b, a_2 + \b], 
\label{Olli.10}
\vspace{2pt} \\
\| u_j - v_j \|_a 
& \leq K_3 \| g \|_\b' \, \theta_j^{a-\a} \quad && \forall a \in [0, a_2].
\label{Olli.11}
\end{alignat}
For $j=0$, \eqref{Olli.10} and \eqref{Olli.11} are trivially satisfied, 
and \eqref{Olli.9} follows from \eqref{Olli.7} because $h_0 = \Psi(0)g_0$ and $\theta_0 = 1$.

Now assume that \eqref{Olli.9}, \eqref{Olli.10}, \eqref{Olli.11} hold for $j=0,\ldots,k$, 
for some $k \geq 0$. 
First we prove \eqref{Olli.11} for $j=k+1$.
Since $u_{k+1} = \sum_{j=0}^k h_j$, 
the definition of the norm of $E_\a'$ and \eqref{Olli.9} for $j=0,\ldots,k$ imply that
$ % \begin{equation} \label{new.3}
\| u_{k+1} \|_\a' \leq K_1 \| g \|_\b'.
$ % \end{equation}
By \eqref{S3 plus} one has
\begin{equation} \label{new.4}
\| u_{k+1} - v_{k+1} \|_0 \leq C K_1 \| g \|_\b' \theta_{k+1}^{-\a}
\end{equation}
where the constant $C$ depends on $\a$. 
From now until the end of this proof we denote by $C$ 
any constant (possibly different from line to line) 
depending only on $a_1, a_2, \a, \b, \mu, a_0$, which are fixed parameters.
From \eqref{Olli.9} with $j=0,\ldots,k$ we get
\begin{equation} \label{new.5}
\| u_{k+1} \|_a \leq K_1 \| g \|_\b' \sum_{j=0}^k \theta_j^{a-\a-1} \quad 
\forall a \in [a_1, a_2].
\end{equation}
We note that 
\begin{equation} \label{new.7}
\sum_{j=0}^k \theta_j^{p-1} \leq \frac{2}{p} \, \theta_{k+1}^p \quad
\forall p > 0. 
\end{equation}
For $a = a_2$, by \eqref{S1} one gets $\| v_{k+1} \|_{a_2} \leq C \| u_{k+1} \|_{a_2}$. 
Thus, using \eqref{new.7} at $p = a_2 - \a$,
\begin{equation} \label{new.6}
\| u_{k+1} - v_{k+1} \|_{a_2} \leq C \| u_{k+1} \|_{a_2} 
\leq C K_1 \| g \|_\b' \theta_{k+1}^{a_2-\a}.
\end{equation}
Using \eqref{S5} to interpolate between \eqref{new.4} and \eqref{new.6}, 
we get \eqref{Olli.11} for $j=k+1$, for all $a \in [0,a_2]$, 
provided that $K_3 \geq C K_1$. 

To prove \eqref{Olli.10} for $j=k+1$, we use \eqref{S2}, \eqref{new.5} 
and \eqref{new.7} and we get
\[
\| v_{k+1} \|_a 
\leq C \theta_{k+1}^{a - a_1 - \b} \| u_{k+1} \|_{a_1 + \b}
\leq C \theta_{k+1}^{a - a_1 - \b} K_1 \| g \|_\b' \sum_{j=0}^k \theta_j^{a_1 + \b -\a - 1}
\leq C K_1 \| g \|_\b' \, \theta_{k+1}^{a -\a}
\]
for all $a \in [a_1 + \b, a_2 + \b]$. 
This gives \eqref{Olli.10} for $j = k+1$ provided that 
$K_2 \geq C K_1$.

To prove \eqref{Olli.9} for $j=k+1$, we begin with proving that 
\begin{equation} \label{new.8}
\| y_{k+1} \|_b \leq C K_1 (K_1 + K_3) \| g \|_\b'^2 \, \theta_{k+1}^{b-\b-1} 
\quad \forall b \in [0,a_2 + \b - \a].
\end{equation}
Since $u_j, v_j, u_j + h_j$ belong to $V$ for all $j = 0,\ldots, k$, 
we use Taylor formula and \eqref{Phi sec} to deduce that, for $j = 0, \ldots, k$ 
and $a \in [0, a_2 - \mu]$,
\begin{align} \label{new.9}
\| e_j \|_a 
& \leq C ( \| h_j \|_{a_0} \| h_j \|_{a+\mu} 
+ \| u_j \|_{a+\mu} \| h_j \|_{a_0}^2 
+ \| h_j \|_{a_0} \| v_j - u_j \|_{a+\mu}
\notag \\ & \qquad \ 
+ \| h_j \|_{a + \mu} \| v_j - u_j \|_{a_0}
+ \| u_j \|_{a + \mu} \| h_j \|_{a_0} \| v_j - u_j \|_{a_0} ).
\end{align}
Hence at $j = k$, using \eqref{S2} and then \eqref{new.9}, we have
\begin{align} \label{new.10}
\| S_{\theta_{k+1}} e_k \|_{a_2 + \b - \a} 
& \leq C \theta_{k+1}^p \| e_k \|_{a_2 + \b - \a - p} 
\notag \\ & 
\leq C \theta_{k+1}^p ( \| h_k \|_{a_0} \| h_k \|_{q} 
+ \| u_k \|_{q} \| h_k \|_{a_0}^2 
+ \| h_k \|_{a_0} \| v_k - u_k \|_{q}
\notag \\ & \qquad \quad
+ \| h_k \|_{q} \| v_k - u_k \|_{a_0}
+ \| u_k \|_{q} \| h_k \|_{a_0} \| v_k - u_k \|_{a_0})
\end{align}
where $p := \max \{ 0, \b - \a + \mu \}$ 
and $q := a_2 + \b - \a - p + \mu$. 
Note that $a_2 + \b - \a - p \geq 0$ because $a_2 \geq \mu$.
Since $q \leq a_2$, using also \eqref{new.7} we have
\begin{equation} \label{new.11}
\| u_k \|_q \leq \| u_k \|_{a_2} 
\leq \sum_{j=0}^{k-1} \| h_j \|_{a_2}
\leq K_1 \| g \|_\b' \sum_{j=0}^{k-1} \theta_j^{a_2-\a-1}
\leq C K_1 \| g \|_\b' \, \theta_k^{a_2-\a} \,.
\end{equation}
By \eqref{new.11}, \eqref{Olli.9}, \eqref{Olli.11}, and since $a_0 \leq a_1$, 
the bound \eqref{new.10} implies that
\[ % begin{equation} \label{new.12}
\| S_{\theta_{k+1}} e_k \|_{a_2 + \b - \a} 
\leq C K_1 (K_1 + K_3) \| g \|_\b'^2 \,
\theta_{k+1}^p ( \theta_k^{a_1 + q - 2\a - 1} + \theta_k^{a_2 + 2a_1 - 3\a - 1} )
\] % end{equation}
provided that $K_1 \| g \|_\b' \leq 1$. 
We assume that 
\begin{equation} \label{new.20}
K_1 \| g \|_\b' \leq 1.	
\end{equation}
Both the exponents $(a_1 + q - 2\a -1)$ and $(a_2 + 2a_1 - 3\a - 1)$
are $\leq (a_2 - \a - 1 - p)$ because $a_1 < \a$ and $a_1 + \b + \mu \leq 2\a$. 
Thus 
\begin{equation} \label{new.13}
\| S_{\theta_{k+1}} e_k \|_{a_2 + \b - \a} 
\leq C K_1 (K_1 + K_3) \| g \|_\b'^2 \, \theta_{k+1}^{a_2 - \a - 1}.
\end{equation}
Now we estimate $\| S_{\theta_{k+1}} e_k \|_0$. 
Since $a_0, \mu \leq a_1$, by \eqref{S1} and \eqref{new.9} we get
\begin{equation} \label{new.14}
\| S_{\theta_{k+1}} e_k \|_0 
\leq C \| e_k \|_0
\leq C (1 + \| u_k \|_\mu) 
( \| h_k \|_{a_1}^2 + \| h_k \|_{a_1} \| v_k - u_k \|_{a_1} ).
\end{equation}
By \eqref{Olli.9} and \eqref{new.20},
\begin{equation} \label{new.15}
\| u_k \|_\mu 
\leq \| u_k \|_{a_1}
\leq \sum_{j=0}^{k-1} \| h_j \|_{a_1} 
\leq K_1 \| g \|_\b' \sum_{j=0}^\infty \theta_j^{a_1-\a-1}
= C K_1 \| g \|_\b'
\leq C.
\end{equation}
We use \eqref{Olli.9}, \eqref{Olli.11} and \eqref{new.15} in \eqref{new.14},
and the bound $\theta_{k+1}^{2a_1 - 2\a - 1} \leq \theta_{k+1}^{-\b-1}$, 
to deduce that
\begin{equation} \label{new.16}
\| S_{\theta_{k+1}} e_k \|_0 
\leq C K_1 (K_1 + K_3) \| g \|_\b'^2 \, \theta_{k+1}^{-\b - 1}. 
\end{equation}
Using \eqref{S5} to interpolate between \eqref{new.13} and \eqref{new.16} we obtain
\begin{equation} \label{new.17}
\| S_{\theta_{k+1}} e_k \|_b 
\leq C K_1 (K_1 + K_3) \| g \|_\b'^2 \, \theta_{k+1}^{b -\b - 1}
\quad \forall b \in [0,a_2 + \b - \a]. 
\end{equation}

Now we estimate the other terms in $y_{k+1}$ (see \eqref{new.1}). 
By \eqref{Rj}, \eqref{new.9}, \eqref{Olli.9}, \eqref{Olli.11} and \eqref{new.7},
\begin{align} \label{new.18}
\sum_{i=0}^{k-1} \| R_k e_i \|_b
& \leq \sum_{i=0}^{k-1} C \theta_k^{b - a_2 + \mu - 1} \| e_i \|_{a_2 - \mu}
\notag \\ & 
\leq C K_1 (K_1 + K_3) \| g \|_\b'^2 \, \theta_k^{b - a_2 + \mu - 1}
\sum_{i=0}^{k-1} \theta_i^{a_1 + a_2 -2\a - 1}
\end{align}
for all $b \in [0, a_2 + \b - \a]$.
Since $a_1 + a_2 -2\a > 0$, we apply \eqref{new.7} to the last sum in \eqref{new.18}.
Then, recalling that $\theta_k / \theta_{k+1} \in [\frac12, 1]$, and using the bound 
$a_1 + \b + \mu \leq 2\a$, we deduce that 
\begin{align} \label{new.19}
\sum_{i=0}^{k-1} \| R_k e_i \|_b
\leq C K_1 (K_1 + K_3) \| g \|_\b'^2 \, \theta_{k+1}^{b - \b - 1} 
\quad \forall b \in [0, a_2 + \b - \a].
\end{align}
The sum of \eqref{new.17} and \eqref{new.19} completes the proof of \eqref{new.8}.

Now we are ready to prove \eqref{Olli.9} at $j=k+1$. 
By \eqref{S1} and \eqref{new.5} we have
$\| v_{k+1} \|_{a_1} \leq C \| u_{k+1} \|_{a_1} \leq C K_1 \| g \|_\b'$,
and we assume that $C K_1 \| g \|_\b' \leq \d_1$, so that $\Psi(v_{k+1})$ is defined.
By \eqref{Olli.8}, \eqref{tame in NM}, \eqref{Olli.7}, \eqref{new.8}, \eqref{Olli.10}
one has, for all $a \in [a_1, a_2]$,
\begin{equation} \label{new.21}
\| h_{k+1} \|_a \leq C \| g \|_\b' \{ 1 + (K_1 + K_3) K_1 \| g \|_\b' \} \, \theta_{k+1}^{a - \a - 1}
\end{equation}
provided that $K_2 \| g \|_\b' \leq 1$.
Bound \eqref{new.21} implies \eqref{Olli.9} provided that 
$C \{ 1 + (K_1 + K_3) K_1 \| g \|_\b' \} \leq K_1$.

The induction proof of \eqref{Olli.9}, \eqref{Olli.10}, \eqref{Olli.11} is complete if
$K_1, K_2, K_3, \| g \|_\b'$ satisfy 
\[ 
K_3 \geq C_0 K_1, \ \   
K_2 \geq C_0 K_1, \ \ 
C_0 K_1 \| g \|_\b' \leq 1, \ \ 
K_2 \| g \|_\b' \leq 1, \ \ 
C_0 \{ 1 + (K_1 + K_3) K_1 \| g \|_\b' \}	\leq K_1
\] 
where $C_0$ is the largest of the constants appearing above. 
First we fix $K_1 \geq 2 C_0$. 
Then we fix $K_2$ and $K_3$ larger than $C_0 K_1$, 
and finally we fix $\d_0 > 0$ such that the last three inequalities hold for all $\| g \|_\b' \leq \d_0$.
This completes the proof of \eqref{Olli.9}, \eqref{Olli.10}, \eqref{Olli.11}.

Bound \eqref{Olli.9} implies that the sequence $(u_k)$ converges 
in $E_a$ for all $a \in [0,\a)$. We call $u$ its limit.
Since $u = \sum_{j=0}^\infty h_j$ and each term $h_j$ satisfies \eqref{Olli.9}, 
it follows that $u \in E_\a'$ and $\| u \|_\a' \leq K_1 \| g \|_\b'$ 
by the definition of the norm in $E_\a'$. 

Finally, we prove the convergence of the Nash-Moser scheme. 
By \eqref{new.1} and \eqref{new.24} one proves by induction that
\[ 
\sum_{j=0}^k (e_j + y_j) = e_k + r_k, \quad \text{where} \ \ 
r_k := (I - S_{\theta_k}) \sum_{j=0}^{k-1} e_j,
\quad \forall k \geq 1.	
\] 
Hence, by \eqref{Olli.8} and \eqref{new.2}, 
recalling that $\Phi'(v_j) \Psi(v_j)$ is the identity map,
one has
\[ 
\Phi(u_{k+1}) - \Phi(u_0)
= \sum_{j=0}^{k}[\Phi(u_{j+1}) - \Phi(u_j)] 
= \sum_{j=0}^{k} (e_j + g_j + y_j) 
= G_k + e_k + r_k
\]
where $G_k := \sum_{j=0}^{k} g_j$. 
By \eqref{Olli.7}, $\| G_k - g \|_b \to 0$ as $k \to \infty$, for all $b \in [0,\b)$.
Let $a \in [a_1 - \mu, \a - \mu)$. 
By \eqref{new.5} and \eqref{new.20} we get $\| u_j \|_{a+\mu} \leq C$. 
By \eqref{new.9}, \eqref{Olli.9} and \eqref{Olli.11} we deduce that
\begin{equation} \label{new.26}
\| e_j \|_a \leq C K_1 (K_1 + K_3) \| g \|_\b'^2 \, \theta_j^{a_1 + a + \mu - 2\a - 1}.
\end{equation}
Hence $\| e_k \|_a \to 0$ as $k \to \infty$ because $a_1 + a + \mu - 2\a < 0$, 
and, moreover, $\sum_{j=0}^\infty \| e_j \|_a$ converges. 
By \eqref{S3} and \eqref{new.26}, for all $\rho \in [0,a)$ we have 
\begin{equation}
\| r_k \|_\rho 
\leq \sum_{j=0}^{k-1} \| (I - S_{\theta_k}) e_j	\|_\rho
\leq C \sum_{j=0}^{k-1} \theta_k^{\rho - a} \| e_j	\|_a
\leq C \theta_k^{\rho - a},
\end{equation}
so that $\| r_k \|_\rho \to 0$ as $k \to \infty$.
We have proved that $\| \Phi(u_k) - \Phi(u_0) - g \|_\rho \to 0$ as $k \to \infty$ 
for all $\rho$ in the interval $0 \leq \rho < \min \{ \a - \mu, \b \}$. 
Since $u_k \to u$ in $E_a$ for all $a \in [0,\a)$, 
it follows that $\Phi(u_k) \to \Phi(u)$ in $F_b$ for all $b \in [0,\a - \mu)$.
The theorem is proved.
\end{proof}

\section{Appendix C. Tame estimates}
\label{sec:tame}

In this appendix we recall classical tame estimates for products, compositions of functions 
and changes of variables which are repeatedly used in the paper. 
Recall the notation \eqref{i7} for functions $u(x)$, $x \in \T$, 
in the Sobolev space $H^s := H^s(\T,\R)$.

\begin{lemma} \label{lemma:tame basic}
Let $s_0, s_1, s_2, s$ denote nonnegative real numbers, with $s_0 > 1/2$.
There exist positive constants $C_s$, $s \geq s_0$, 
%For all real numbers $s \geq s_0$ there are constants $C_s$ 
%%There exists an increasing function $s \mapsto C_s$, $s \geq s_0$, 
with the following properties.

\noindent
\emph{(Embedding and algebra)} 
For all $u,v \in H^{s_0}$, 
\begin{equation} \label{1605.2}
\| u \|_{L^\infty} \leq C_{s_0} \| u \|_{s_0}, \quad 
\| uv \|_{s_0} \leq C_{s_0} \| u \|_{s_0} \| v \|_{s_0}.
\end{equation}
\emph{(Interpolation)} 
For $0 \leq s_1 \leq s \leq s_2$ and $s = \lm s_1 + (1-\lm) s_2$, for all $u \in H^{s_2}$,
\begin{equation} \label{interpolation GN}
\| u \|_{s} \leq  \| u \|_{s_1}^\lm \| u \|_{s_2}^{1-\lm}. 
\end{equation}
\emph{(Tame product)} 
For $s \geq s_0$, for all $u,v \in H^s$, 
\begin{equation} \label{asymmetric tame product}
\| uv \|_s \leq C_{s_0} \|u\|_s \|v\|_{s_0} + C_s \|u\|_{s_0} \| v \|_s, 
\end{equation}
and for $s \in [0,s_0]$, for all $u \in H^{s_0}$, $v \in H^s$, 
\begin{equation} \label{1605.1}
\| uv \|_s \leq C_{s_0} \| u \|_{s_0} \|v\|_s.
\end{equation}
\end{lemma} 

\begin{proof} 
The lemma can be proved by using Fourier series and H\"older inequality. 
Otherwise, for \eqref{interpolation GN} see, e.g., \cite{AG} (page 82)
or \cite{Moser-Pisa-66} (p.\,269); 
for \eqref{asymmetric tame product} adapt \cite{BBPro} (appendix)
or \cite{AG} (p.\,84).
For \eqref{1605.1} use the bound 
$\sum_{j \in \Z} \la n \ra^{2s} \la j \ra^{-2s} \la n-j \ra^{-2s_0} \leq C_{s_0}$ 
for all $n \in \Z$, all $0 \leq s \leq s_0$, which can be proved 
by splitting the two cases $2|j| \leq |n|$ and $2|j| > |n|$.
\end{proof}

A function $f : \T \times B \to \R$, where $B := \{ y \in \R^{p+1} : |y| < R \} $, 
induces the composition operator 
\begin{equation} \label{comp}
\tilde f(u)(x) := f(x, u(x), u'(x), u''(x), \ldots, u^{(p)}(x)) 
\end{equation}
where $u^{(k)}(x)$ denotes the $k$-th derivative of $u(x)$. 
Let $B_p$ be a ball in $W^{p,\infty}(\T,\R)$ such that, if $u \in B_p$, then 
the vector $(u(x), u'(x), \ldots, u^{(p)}(x))$ belongs to $B$ for all $x \in \T$.

\begin{lemma}[Composition of functions]
\label{lemma:tame cf}
Assume $ f \in C^r (\T \times B)$. 
Then, for all $ u \in H^{s+p} \cap B_p$, $s \in [0,r]$, 
the composition operator \eqref{comp} is well defined and 
\[
\| \tilde f(u) \|_s
\leq C \| f \|_{C^r} (\|u\|_{s+p} + 1) 
\]
where $C$ depends on $r,p$. 
If, in addition, $f \in C^{r+2}$, then, for $u,h \in H^{s+p}$ with 
$u, u+h \in B_p$, one has  
\begin{align*}
\big\| \tilde f(u+h) - \tilde f (u) \big\|_s 
& \leq C \| f \|_{C^{r+1}} \, ( \| h \|_{s+p} + \| h \|_{W^{p,\infty}} \| u \|_{s+p}) \,,  
\\
\big\| \tilde f(u+h) - \tilde f (u) - \tilde f'(u) [h] \big\|_s 
& \leq C \| f \|_{C^{r+2}} \, \| h \|_{W^{p,\infty}} 
(\| h \|_{s+p} + \| h \|_{W^{p,\infty}} \| u \|_{s+p}). 
\end{align*}
\end{lemma}

\begin{proof}
For $s \in \N$ see \cite{Moser-Pisa-66} (p.\,272--275) 
and \cite{Rabinowitz-tesi-1967} (Lemma 7, p.\,202--203).
For $s \notin \N$ see \cite{AG} (Proposition 2.2, p.\,87).
\end{proof}

%The next lemma is also classical, see for example % \cite{Hormander-geodesy}, Appendix, and 
%\cite{Ioo-Plo-Tol}, Appendix G. 
%The present version is proved in \cite{Baldi-Benj-Ono}, 
%adapting Lemma 2.3.6 on page 149 of \cite{Hamilton}. 

\begin{lemma}[Change of variable]  
\label{lemma:tame cv} 
Let $p \in W^{s,\infty}(\T,\R)$, 
$s \geq 1$, with $\| p \|_{W^{1,\infty}} \leq 1/2 $. 
Let $f(x) = x + p(x)$. 
Then $f$ is invertible, its inverse is $f\inv(y) = g(y) = y + q(y)$ where $q$ is $ 2 \pi $-periodic, $q \in W^{s,\infty}(\T,\R)$, and 
$\| q \|_{W^{s, \infty}} \leq C \| p \|_{W^{s, \infty}}$,
where $C$ depends on $d, s$. 

Moreover, if $u \in H^s(\T,\R)$, then $u \circ f(x) = u(x+p(x))$ also belongs to $H^s$, 
and 
\begin{equation}
\| u \circ f \|_s + \| u \circ g \|_s
\leq  C (\|u\|_s + \| p \|_{W^{s, \infty}} \|u\|_1).
\label{tame-cambio-di-variabile}
\end{equation}
\end{lemma}

\begin{proof} 
For $s \in \N$ see, e.g., \cite{Baldi-Benj-Ono} (Lemma B.4 in the appendix),
where this lemma is proved by adapting \cite{Hamilton} (Lemma 2.3.6, p.\,149).
For $s \notin \N$ the lemma can be proved by studying the conjugate of the pseudo-differential operator $|D_x|^s$ by a change of variable, either by Egorov's Theorem, 
see \cite{Taylor} (ch.\,VIII, sec.\,1, p.\,150) and \cite{ABH} (appendix C, sec.\,C.1),
or by asymptotic formula, see \cite{AG} (Proposition 7.1, p.\,37).
\end{proof}

\begin{remark} For time-dependent functions $u(t,x)$, 
$u \in C([0,T], H^s(\T,\R))$, all the estimates of the present appendix hold 
with $\| u \|_s$ replaced by $\| u \|_{T,s} := \sup_{t \in [0,T]} \| u(t) \|_{s}$.
\end{remark}

\bigskip

\begin{flushright}

Pietro Baldi, Giuseppe Floridia, Emanuele Haus 

\medskip

Dipartimento di Matematica e Applicazioni ``R. Caccioppoli''

Universit\`a di Napoli Federico II  

Via Cintia, 80126 Napoli, Italy

\medskip

\texttt{pietro.baldi@unina.it} 

\texttt{giuseppe.floridia@unina.it} 

\texttt{emanuele.haus@unina.it}
\end{flushright}

\end{document}